\documentclass[11pt,a4paper]{amsart}

\usepackage{extsizes}
\usepackage{enumerate}
\usepackage{amsmath}
\usepackage{amsthm}
\usepackage{amssymb}
\usepackage{amsbsy}
\usepackage{amsfonts}
\usepackage{amstext}
\usepackage{amscd}
\usepackage{tikz}
\usepackage{graphicx}
\usetikzlibrary{shapes.geometric}
\numberwithin{equation}{section}
\theoremstyle{plain}
\newtheorem{thm}{Theorem}[section]
\newtheorem{prop}[thm]{Proposition}
\newtheorem{cor}[thm]{Corollary}
\newtheorem{lem}[thm]{Lemma}
\theoremstyle{definition}
\newtheorem{exa}[thm]{Example}
\newtheorem{conj}[thm]{Conjecture}
\newtheorem{rem}[thm]{Remark}
\newtheorem{defi}[thm]{Definition}

\newcommand{\Mod}[1]{\ (\mathrm{mod}\ #1)}

\begin{document}
\title [Vertex Energy] {Energy of a Vertex}



\author[Arizmendi]{Octavio Arizmendi}
\author[Fernandez Hidalgo]{Jorge Fernandez Hidalgo}
\author[Juarez-Romero]{Oliver Juarez-Romero}

\address{Centro de Investigaci{\'o}n en Matem{\'a}ticas. Guanajuato, M\'exico}
\email{octavius@cimat.mx}

\address{Escuela Normal Superior Oficial de Guanajuato.  Guanajuato,  M\'exico}
\email{ojuarez@cimat.mx}

\address{Universidad Nacional Aut{\'o}noma de M{\'e}xico, M{\'e}xico}
\email{jorgefernandez@ciencias.unam.mx}
\date{\today}

\maketitle
\begin{abstract}
In this paper we develop the concept of energy of a vertex introduced in Arizmendi and Juarez-Romero (2016). We derive basic inequalities, continuity properties, give examples and extend the definition to locally finite graphs.
\end{abstract}

\tableofcontents

\addtocontents{toc}{\protect\setcounter{tocdepth}{1}}

\section{Introduction and Statements of Results}

The {\it graph energy} is a graph invariant that was defined by I. Gutman \cite{Gut2} in his studies of mathematical chemistry.
 Specifically, this concept emerged from the application of H\"uckel Molecular Orbital (HMO) theory to the study of conjugated hydrocarbons in theoretical chemistry.   An excellent introduction to the theory of graph energy can be found in the monograph \cite{LiShiGu}, see also \cite{Gu}. Formally, the energy of a graph $G$, denoted by $\mathcal{E}(G)$, is defined as the sum of the absolute values of the eigenvalues of the adjacency matrix $A = A(G)$, i.e.,
\begin{equation*}
  \mathcal{E}(G) = \sum_{i=1}^{n}|\lambda_{i}|.
\end{equation*}

In a recent paper by the first and third author, a refinement for vertices for the energy was defined, which we will call the \emph{energy of a vertex}\footnote{ Not to be confused with the concept of \emph{energy per vertex} defined by van Dam et. al. \cite{vanHaeKoo} which is the energy of the graph divided by the number of vertices. }.  Precisely, for a graph $G=(V,E)$, with vertex set ${v_1,2,\cdots,v_n}$,  the energy of the vertex $v_i$ with respect to $G$, which is denoted by $\mathcal{E}_G(v_i)$, is given by
\begin{equation*}
  \mathcal{E}_G(v_i)=|A|_{ii}, \quad\quad~~~\text{for } i=1,\dots,n,
\end{equation*}
where $|A|=(AA^*)^{1/2}$ and $A=A(G)$ is the adjacency matrix of $G$.

The energy of a vertex should be understood as the contribution of this vertex to the energy of the graph, in terms of how it interacts with other vertices. It can be seen that the energy of a vertex only depends on the vertices that are in the same component as $v$. To be more precise, $\mathcal{E}_G(v_i)=\mathcal{E}_{C(v_i)}(v_i),$ where $C(v_i)$ is the connected component of $v_i$.  In particular, one recovers the (trivial) fact that if $v$ is an isolated vertex then it does not contributes to the energy.

As observed by Nikiforov the energy of a graph is given by the nuclear (or $L^1$) norm of the adjacency matrix,  $\mathcal{E}(G)=Tr(|A(G)|)$. It follows that we can calculate the energy of a graph by adding the individual energies of the vertices of $G$,
\begin{equation*}
  \mathcal{E}(G)=\mathcal{E}_G(v_1)+\cdots+\mathcal{E}_G(v_n).
\end{equation*}

In this paper we continue the development of the concept of energy of a vertex, following different directions:

\begin{enumerate}
\item We present a direct way to calculate the energy of a vertex in terms of the eigenvalues and eigenvectors of the graph and in terms of the distribution of a graph with respect to a vertex, see Section \ref{EV}.
\item We deduce inequalities for the energy of a vertex,  which in some cases improve dramatically the known results for the graph energy (Section \ref{Bounds}). This includes a version for vertices of the Koolen-Moulton Inequality \cite{KM}, see Section \ref{Koolen-MoultonIN}.
\item In Section \ref{HypoHyper},  we introduce the notions of hyperenergy and hypoenergy for vertices and give sufficient criteria to ensure any of them.
\item We give examples and counterexamples of natural conjectures for the energy of a vertex (Section \ref{ExaCounter}).
\item Finally, in Section \ref{LFG}, we extend the definition and inequalities to locally finite graphs.
\end{enumerate}

We hope that we are able to show the usefulness of this refinement for the general theory of graph energy.

\section{Energy of a vertex}\label{EV}
In this section we recall the notion of energy of a graph. Before doing this we give a brief summary of the preliminaries needed.

\subsection{Preliminaries on graphs and matrices}

In this work we will work mostly work with simple undirected finite graphs.  To be precise a graph $G$ is a pair $G=(V(G),E(G))$, where $E(G)\subset V(G)\times V(G)$.  $V:=V(G)$ is called the vertex set $V(G)$ and $E:=E(G)$ the edge set. The graph $G$ is undirected if $(v,w)\in E$ implies that $(w,v)\in E$ and simple if there are no vertices of the form $(v,v).$

A finite graph  $G=(V(G),E(G))$ is said to be of order $n$ and size $m$ if $n=|V(G)|$ and $m=|E(G)|$. If $G$ is a graph of order $n$ we label the vertices of the graph $G$ as $v_{1},v_{2}, \ldots, v_{n}$.  The vertices $v_{i}$ and $v_{j}$ are adjacent in $G$ if the edge $(v_{i},v_{j})$ is in $E(G)$; in this case we write $v_i\sim v_j$
. For a vertex $v_{i} \in G$, the degree of $v_{i}$ is the number of adjacent vertices to $v_i$ and denoted $d_{i}= deg(v_{i})$. A vertex of degree $1$ is called a leaf vertex (or simply, a leaf); sometimes, it is also called a pendent vertex. The minimum degree of $G$ is denoted by $\delta:=\delta(G)$, while the maximum degree is denoted by $\Delta:=\Delta(G)$. If  A $d$-regular graph  $G$ is a graph where $\Delta=\delta=d$.

In the last section we will also consider uniformly locally finite graphs. That is, graphs $G$ such that
\begin{equation*}
\sup_{v\in G}{\deg(v)}=\Delta<\infty.
\end{equation*}

 Walks on the graph  will be important to understand the combinatorial properties of the energy.  A walk of length $k$ in $G$ is a sequence of vertices $v_{i_{1}},v_{i_{2}},\cdots ,v_{i_{k}}$ such that $(v_{i_{r}},v_{i_{r+1}})$ is in $E(G)$ for $r=1,2,\ldots,k-1$. We say that a walk is a closed walk if $v_{i_{1}} = v_{i_{k}}$.  A path is a walk where all vertices are distinct.
 A cycle is closed walk with $v_{i_k}\neq v_{i_l}j$ for $\leq l,k\leq n$ . If there is a path between  the vertices $v_{i}$  and $v_{j}$ we say that the vertices are connected.  The distance $d(v_{i},v_{j})$ is the length of a shortest path between $v_{i}$ and $v_{j}$. By using this distance, we define the diameter $d(G)$ as the maximum distance between any pair of vertices, the eccentricity $\epsilon(v_{i})$ as the maximum distance between $v_{i}$ and any other vertex, while the radius $r(G)$ as the minimum eccentricity of the vertices in the graph $G$.

Given two graphs $G$ and $H$, we say that $H$ is a subgraph of $G$ if $V(H)\subseteq V(G)$ and $E(H)\subseteq E(G)$. A component $C$ of $G$ is a subgraph where every pair of vertices in $V(C)$ are connected.  We denoted by $C(v_i)$ the connected component that contains  the vertex $v_i$.  A tree $T$ is graph that is connected and has no cycles. A bipartite graph $G$ is a graph where there are two set of vertices $V_1, V_2\subset V(G)$, called the parts of graph, satisfying the following conditions $(i)$ $V = V_1 \cup V_2$ with $V_1 \cap V_2 = \emptyset$, ($ii$) every edge connects a vertex in $V_1$ with one in $V_2$.  If $p=|V_1|$ and $q=|V_2|$, the graph $G$ is called a $p,q$-bipartite graph.

The adjacency matrix $\textbf{A}=\textbf{A}(G)$ of $G$ is a square matrix of order $n$ whose $(i,j)$-entry is defined as
\begin{equation}
   A_{ij} = \begin{cases} 1 &\mbox{if} v_{i} \sim v_{j},\\
0 & \mbox{otherwise.} \end{cases}
\end{equation}

The eigenvalues of $A(G)$ are said the eigenvalues of the graph $G$. A graph on $n$ vertices has $n$ eigenvalues counted with multiplicity; these will be denoted by $\lambda_{1},\lambda_{2}, \ldots, \lambda_{n}$ and labeled in a decreasing manner: $\lambda_{1} \geq \lambda_{2} \geq \cdots \geq \lambda_{n}$. The set of the all $n$ eigenvalues of $G$ is also called the spectrum of $G$.
Since $A(G)$ is  self-adjoint then the eigenvalues of the graph $G$ are necessarily real-valued.
For more details about graph spectrum see \cite{BH} and \cite{CDS}.

\subsection{Definition and first properties}

In this section we present the definition of the energy of a vertex and give a precise way to calculate it in terms of the spectrum and eigenvectors. We also state some basic properties.

Let us consider a graph $G=(V,E)$ with vertex set $V=\{v_1,...,v_n\}$ and adjacency matrix $A\in M_n(\mathbb{C})$. If for a matrix $M$, we denote its trace by $Tr(M)$, and its absolute value $(MM^*)^{1/2}$, by $|M|$, then the energy of $G$ is given by
\begin{equation*}
\mathcal{E}(G)=Tr(|A(G)|)=\displaystyle\sum_{i=1}^{n}|A(G)|_{ii}.
\end{equation*}
With this in mind, the authors in \cite{ArJu} defined the energy of a vertex as follows.
\begin{defi}
The energy of the vertex $v_i$ with respect to $G$, which is denoted by $\mathcal{E}_G(v_i)$, is given by
\begin{equation}
  \mathcal{E}_G(v_i)=|A(G)|_{ii}, \quad\quad~~~\text{for } i=1,\dots,n,
\end{equation}
where $|A|=(AA^*)^{1/2}$ and $A$ is the adjacency matrix of $G$.
\end{defi}
The idea behind this definition comes from the fact that  the trace can be decomposed as the sum of the \textbf{positive} linear functionals $\phi_i:M_n\to\mathbf{C}$, defined by $\phi_i(A) \mapsto A_{ii},$
\begin{equation}
  Tr(A(G))=\phi_1(A(G))+\cdots+\phi_n(A(G)).
\end{equation}

In this way the energy of a graph is given by the sum of the individual energies of the vertices of $G$,
\begin{equation}
  \mathcal{E}(G)=\mathcal{E}_G(v_1)+\cdots+\mathcal{E}_G(v_n),
\end{equation}
and thus we believe that the energy of a vertex should be understood as the contribution of this vertex to the energy of the graph, in terms on how it interacts with other vertices.

Positivity of the linear functional is very useful to give inequalities. In particular, for $\phi_i$, being a positive linear functional on $M_n$ implies that the H\"older inequality is also valid. More precisely, let $0<r,p,q\leq\infty$ be such that  $1/r=1/p+1/q$ then
\begin{equation}\label{Equa197}
\phi_i(|AB|^r)\leq \phi_i(|A|^p)^{1/p}\phi_i(|B|^q)^{1/q}.
\end{equation}
In particular, for $p=2$ and $q=2$,  since $|\phi_i(AB)|\leq \phi_i(|AB|)$ we get the Cauchy-Schwarz inequality,
\begin{equation}\label{Equa198}
   |\phi_i(AB)|\leq \phi_i(AA^*)^{1/2}\phi_i(BB^*)^{1/2}.
\end{equation}

The following lemma tells us how to calculate the energy of a vertex in terms of the eigenvalues and eigenvectors of $A$.
\begin{lem}\label{W1} Let $G=(V,E)$ be a graph with vertices $v_1,...,v_n$. Then
\begin{equation}\label{Equa2}
  \mathcal{E}_{G}(v_i)=\displaystyle\sum_{j=1}^{n}p_{ij}|\lambda_j|,\quad i=1,\ldots,n
\end{equation}
where $\lambda_j$ denotes the $j$-eigenvalue of the adjacency matrix of $A$ and the weights $p_{ij}$ satisfy
$$\sum^n_{i=1}p_{ij}=1\text{ and } \sum^n_{j=1}p_{ij}=1.$$
Moreover, $p_{ij}=u_{ij}^{2}$ where $U = (u_{ij})$ is  the orthogonal matrix whose columns are given by the eigenvectors of $A$.
\end{lem}
\begin{proof} Recall that the matrix $A$ has $n$ linearly independent real eigenvectors, which can in fact be chosen
to be orthonormal (i.e., orthogonal and of unit length). Let $u_1,u_2,\ldots, u_n$ be real orthonormal unit eigenvectors for $A$, with corresponding eigenvalues $\lambda_1,\lambda_2,\ldots, \lambda_n$. If $U = (u_{ij})$  the matrix whose columns are $u_1,u_2,\ldots, u_n$ denoted $U=[u_1,u_2,\ldots, u_n]$ and $\Lambda=(\delta_{ij}\lambda_i)^n_{i,j=1}$, then $U$ is orthogonal (i.e., $U^{-1}=U^{t}$) and
\begin{equation}
A=U\Lambda U^{t}.
\end{equation}
This implies that  $AA^*=U\Lambda \Lambda^* U^{t}$. and then
\begin{equation}\label{Equa3}
|A|=(AA^*)^{1/2}=U(\Lambda \Lambda^*)^{1/2} U^{t}.
\end{equation}
Since $[(\Lambda \Lambda^*)^{1/2}]_{ij}=\delta_{ij}|\lambda_i|$ and the vectors $u_i$ are orthogonal, according to (\ref{Equa3}), we have
\begin{equation}\label{Equa35}
|A|_{ii}=\sum_{j,h=1}^{n} U_{ih}[(\Lambda \Lambda^*)^{1/2}]_{hj} U^t_{ji}=\sum_{j=1}^{n}u_{ij} |\lambda_{j}|u_{ij}. =\displaystyle\sum_{j=1}^{n}u_{ij}^{2}|\lambda_{j}|
\end{equation}
as desired.
\end{proof}

For our treatment it will be convenient to associate a measure to any graph which is rooted at a given vertex. In order to do this, we define, for $k$ in $\mathbb{N}$, the $k$-th moment of $a$ with respect to the linear functional $\phi_i$ by $\phi_i(a^k)$.

Now the distribution can be defined in terms of the moments of $a$;  for any self-adjoint matrix $a\in M_n$ there exists a unique probability measure $\mu_a$  with the same moments as $a$, that is,
\begin{equation*}
   \int_{\mathbb{R}}x^{k}\mu_a (dx)=\phi_i (a^{k}), \quad \forall k\in \mathbb{N}.
\end{equation*}
We call  $\mu$ the distribution of $a$ with respect to $\phi_i$. By a similar argument as in the proof of Lemma \ref{W1} we can see that
\begin{equation}\label{momentos}
\phi_i(a^k)=\displaystyle\sum_{j=1}^{n}p_{ij}\lambda_j^k,\quad i=1,\ldots,n
\end{equation}
where $p_{ij}$ are as in \eqref{Equa2}, and thus we recognize the measure $\mu_a$,
$$\mu_a=\sum_j p_{ij}\delta_{\lambda_j}.$$
We will denote by $M_k(G,i)$ the quantity $\phi_i(A^k)$ when $A$ is the adjacency matrix of the graph $G$. Notice that $M_k(G,i)$ is equal to the number of $v_i-v_i$ walks in $G$ of length $k$ which also coincides with the sum $\sum_j p_{ij}\lambda_j^k$.

\section{Some basic bounds}\label{Bounds}

One basic problem is to find the extremal values or good bounds for the energy within some special class of graphs and to characterize graphs from this class which reach this extremal values of the energy. In this section we present upper and lower bounds for the energy of a vertex which can be derived rather simply from the theory but which turn out to be quite useful. These will be used in the next section when considering hyperenergetic and hypoenergetic graphs.

\subsection{McClelland's upper bound}

 One of the first bounds for the energy of graph was given by McClelland in \cite{McC}. He considered the $n$ vertices and $m$ edges of a graph $G$ for the following upper bound,
\begin{equation}\label{Equa1}
  \mathcal{E}(G) \leq \sqrt{2mn}.
\end{equation}
Improvement of this inequality were given by Koolen and Moulton \cite{KM,KM2}, Zhou \cite{Zhou} and Nikiforov for matrices \cite{Nik} among others.

In this direction, Arizmendi and Juarez-Romero \cite{ArJu} derived an inequality for the energy of a graph, similar to McClelland's.

\begin{thm}\cite{ArJu}
For a graph $G$ with vertices $v_1,...,v_n$ with degrees $d_1,...,d_n$ we have
\begin{equation}\mathcal{E}(G)\leq\sum^n_{i=1} \sqrt{d_i}\leq\sqrt{2mn}.
\end{equation}
The second inequality holds if and only if $G$ is a regular graph.
\end{thm}

The main observation for the first inequality in the  theorem  above is the fact that one can bound the energy by the second moment  w. r. t. to $\phi_i$ and this turns out to coincide with the degree of the vertex.  We give the details of the proof to include the characterization of the graphs for which the equality holds.

\begin{prop}\label{McClelland vertex}
For a graph $G$ and a vertex $v_i\in G$
$$ \mathcal{E}_G(v_i)\leq \sqrt{d_i}$$
with equality if and only if the connected component containing $v_i$ is isomorphic to $S_n$ and $v_i$ is its center.
\end{prop}

\begin{proof}

Since $\phi_i$ is positive we may use the inequality (\ref{Equa198}) from the preliminaries  which read as follows:
\begin{equation} \label{CS2}
\phi_i(|A|)^2\leq\phi_i(|A|^2)=\phi_i(A^2).
\end{equation}
Since $\phi_i(|A|)=\mathcal{E}_G(v_i)$ and $\phi_i(A^2)=d_i)$ the inequality follows.

Now, for the equality to hold, we may assume directly that $G$ is connected.  The equality in (\ref{CS2}) hold only if the distribution of $|A|$ is a dirac mass $\delta_c$ for some $c$. This implies that $|A|$, and thus $A^2$, is the identity. From this, one deduces that all the vertices of $G$ are at distance at most $1$ from $v_i$ and $G$ is a star with center $v_i$. Finally, from (\ref{Star}) below we see the equality holds for the center of the  star.

\end{proof}

\subsection{Lower bounds}

Now we consider a simple lower bound in terms of the degrees of the graph.

\begin{thm}\label{4.1} Let $G$ be a  connected graph with at least one edge. Then
\begin{equation}
\mathcal{E}_G(v_i)\geq \cfrac{d_i}{\Delta}, \quad\quad~~~\text{for all } v_i\in V.
\end{equation}
Equality holds if and only if $G$ is isomorphic to complete bipartite graph $K_{d,d}$.
\end{thm}
\begin{proof}
	Notice $|x| \geq x^2$ for $x\in [-1,1]$ and so  $|\frac{x}{\Delta}| \geq (\frac{x}{\Delta})^2$ for
$x\in [-\Delta,\Delta]$, with equality only at
$\{-\Delta,0,\Delta\}$. Since $\lambda_{1}\leq \Delta$ we conclude
	$E_G(v_i)=\sum\limits_{j=1}^n p_{ij}|\lambda_i|\geq \sum\limits_{j=1}^n p_{ij}\frac{\lambda_j^2}{\Delta}=\frac{d_i}{\Delta}$, equality occurs if and only if $\lambda_j\in\{\Delta,0,\Delta\}$ whenever $p_{ij}>0$. In order for $\Delta$ or $-\Delta$ to be eigenvalues of $G$ it is necessary for $G$ to be a regular graph, and so we know that $p_{i1}=\frac{1}{n}$. If $G$ is not bipartite then $-d$ is not an eigenvalue of the graph, and by calculating the second moment we obtain $p_{i1}d^2=d$ and so $p_{i1}=\frac{1}{d}$, which implies $d=n$ which is impossible. If the graph is bipartite then the spectra of each vertex is symmetric which tells us $2p_{i1}=\frac{1}{d}$, which tells us $n=2d$, and so the graph is the complete regular bipartite graph $K_{d,d}$.

\end{proof}

A more sophisticated bound is found by using H\"older Inequality.

\begin{thm} Let $G$ be a graph of order $n$ with at least one edge, $k\geq2$, $0<p,q<\infty$ and $1=\frac{1}{p}+\frac{1}{q}$. Then
\begin{equation}\label{Equa11}
\cfrac{\left(\phi_{j}(|A|^k)\right)^{q}} {\left(\phi_j(|A|^{p(k-1)+1})\right)^{\frac{q}{p}}} \leq\mathcal{E}_{G}(v_j), \quad \text{ for all } j=1,\dots,n.
\end{equation}
\end{thm}
\begin{proof}
Taking $M=|A|^{k-\frac{1}{q}} \text { and } N=|A|^{\frac{1}{q}}$ in the matricial version of the H\"older Inequality
\begin{equation}
\left(\phi_j(|MN|^r)\right)^{\frac{1}{r}}\leq (\phi_j(|M|^p))^{\frac{1}{p}}(\phi_j(|N|^q))^{\frac{1}{q}},
\end{equation}
where $0<p,q<\infty$ and $1=\frac{1}{p}+\frac{1}{q}$, we have
$$\phi_j(|A|^{k})=\phi_j\left(|A|^{k-\frac{1}{q}}|A|^{\frac{1}{q}}\right)  \leq \left(\phi_j\left(|A|^{k-\frac{1}{q}}\right)^{p}\right)^{\frac{1}{p}}\left(\phi_j\left(|A|^{\frac{1}{q}}\right)^{q}\right)^{\frac{1}{q}}.$$
Which, in turn, taking the $q$-th power, implies
$$
\left(\phi_j(|A|^{k})\right)^{q}  \leq\left(\phi_j\left(|A|^{kp-\frac{p}{q}}\right)\right)^{\frac{q}{p}}\phi_j\left(|A|\right),$$
Finally, since $kp-p/q=p(k-1)+1$,   we conclude that
\begin{equation*}
\cfrac{\left(\phi_{j}(|A|^k)\right)^{q}} {\left(\phi_j(|A|^{p(k-1)+1})\right)^{\frac{q}{p}}} \leq\mathcal{E}_{G}(v_j).
\end{equation*}
\end{proof}

For special formulation of the last theorem see  Gutman et. al. \cite{GFDR} and Zhou et. al. \cite{Zhou}. The following particular case appears very useful since the quantities involved have a combinatorial meaning.

\begin{cor} \label{M_44}Let $G$ be a graph of order $n$ with at least one edge. Then
\begin{equation}\label{M_4}
\cfrac{(d_j)^{3/2}} {M_4(G,j))^{1/2}}\leq \mathcal{E}_{G}(v_j), \quad \text{ for all } j=1,\dots,n.
\end{equation}
\end{cor}
\begin{proof}
Taking $k=2$, $p=3$  and $ q=\frac{3}{2}$ in (\ref{Equa11}) we obtain  $$\cfrac{\phi_{j}(A^2)^{3/2}} {\phi_j(A^{4})^{1/2}}\leq \mathcal{E}_{G}(v_j), \quad \text{ for all } j=1,\dots,n. $$  The conclusion follows since $\phi_{j}(|A|^2)=d_i$ and $\phi_j(|A|^{4})$ counts the number of closed walks of size $4$ from $v_j$, i.e., $\phi_j(A^{4})=M_4(G,j)$.
\end{proof}

Here we note that bounding from above the number of walks of size $4$ one obtains the following improvement of Theorem \ref{4.1}.

\begin{thm}\label{lowerbound2} Let $G$ be a  graph with at least one edge. Then
\begin{equation}
\mathcal{E}_G(v_i)\geq \sqrt{ \cfrac{d_i}{\Delta}} , \quad\quad~~~\text{for all } v_i\in V.
\end{equation}
\end{thm}

\begin{proof}
Set $d_i=d$.
The number of walks of size $4$ can be bounded from above by $d^2+d(\Delta-1)+d(\Delta-1)(d-1)=d^2\Delta.$  Indeed, there are $d_i^2$ walks that visit the vertex $v_i$ twice, this are walks of the form $v_i\sim v_j\sim v_i\sim v_k\sim v_i$. On the other hand , there are at most $d_i(\Delta-1)$ walks of the form $v_i\sim v_j\sim v_k\sim v_j\sim v_i$, with $k\neq i$. Finally the number of walks starting at $v_i$ which form quadrangles (i.e.   $v_i\sim v_j\sim v_k\sim v_l\sim v_i$, with $i\neq k$ and $j\neq l$) is at most $d(\Delta-1)(d-1)$.
 Thus, from Theorem \ref{M_4}, $$\mathcal{E}_G(v_i)\geq \frac{d^{3/2}}{(d^2\Delta)^{1/2}}.$$
 \end{proof}

\begin{rem}  Putting  together Proposition \ref{McClelland vertex} with Theorem \ref{lowerbound2}  we obtain that  for a graph $G$ with $n$ vertices and any vertex $v\in G$,
 $$ \frac{1}{\sqrt{n-1}}\leq \mathcal{E}_G(v)\leq\sqrt{n-1}.$$
 These inequalities become equalities for the leafs and the center of the star $S_n$,  respectively (see Section 5.2).
\end{rem}

In the case of graphs with no quadrangles we can improve dramatically this bound.

\begin{prop} \label{no quadrangles }
If $v_i$ is not contained in any quadrangle of $G$ then
$$ E_G(v_i)\geq \sqrt{d_i} \sqrt{\frac{d_i}{d_i+\Delta-1}}$$
\end{prop}
\begin{proof}

 In this case, since we do not consider quadrangles, there there are at most $d^2+d(\Delta-1)=d_i(d_i+\Delta-1)$ walks  of size 4.  Hence we have $E_G(v_i)\geq \sqrt{ \frac{d_i^3}{d_i(d_i+\Delta-1) }}$ as desired.

\end{proof}

\subsection{Bipartite graphs}

In this section we show and use the nice property that energy of a bipartite graph is divided evenly between its two parts.   More precisely,
\begin{prop}\label{bipartite}
Let G be a bipartite graph with parts $V$ and $W$ then
\begin{equation}
\sum_{v \in V} \mathcal{E}_{G} (v) = \sum_{w \in W} \mathcal{E}_{G} (w).
\end{equation}
\end{prop}

\begin{proof}
Being bipartite corresponds to having a matrix of the form,
\begin{equation}
A =
\begin{pmatrix}
{\bf 0}_{r,r} & B \\
 B^T    & {\bf 0}_{s,s}
\end{pmatrix},
\end{equation}
then the  matrix $M=AA^T$ is of the form
\begin{equation}
M =
\begin{pmatrix}
BB^T & 0\\
 0   & B^T B
\end{pmatrix},
\end{equation}
and the absolute value of the matrix is given by
\begin{equation}
|A| =
\begin{pmatrix}
\sqrt{BB^T} & 0\\
 0   &\sqrt{B^T B}
\end{pmatrix}
\end{equation}
Now, we have the equalities
$$Tr|B^TB|= \sum_{i\in|V_2|} |A|_{ii} =\sum_{i\in|V_2|} E(v_i),$$
and
$$ Tr|BB^T|=\sum_{i\in|V_1|} |A|_{ii} =\sum_{i\in|V_1|} E(v_i),$$
from where, since   $Tr|BB^T|=Tr|B^TB|$, we get the result.
\end{proof}

The above theorem is very useful combined with the inequalities given for the vertex in term of the degree.

\begin{cor} \label{corbipartite}The energy of a bipartite graph is bounded from below and from above as follows
$$\frac{1}{\sqrt{\Delta}}2 \sum_{v \in V_1} \sqrt{ deg (v)}\leq \mathcal{E}(G)\leq  2 \sum_{v \in V_1} \sqrt{ deg (v)}.$$
\end{cor}

Typically a good upper bound using the above corollary will imply a bad lower bound and viceversa. In other words one of the parts give a better if the other one gives a better upper bound.  The following example shows that both inequalities can be sharp.

\begin{exa}  Consider the star $S_n$, if one takes $V_1$ to contain only the center $v_n$ one arrives at the \emph{sharp} upper bound $\mathcal{E}(G)\leq 2deg(v_n)=2\sqrt{n-1}$. Taking $V_1$ as above for the lower bound gives $\mathcal{E}(G)\geq 2 $. However, taking $V_1$ to be the rest of the vertices one arrives at the \emph{sharp} lower bound $\mathcal{E}(G)\geq 2 ( n-1)/\sqrt{n-1}= 2\sqrt{n-1} $.

\end{exa}

\section{Hypoenergetic and Hyperenergetic Graphs}\label{HypoHyper}

In this section we want to revisit the concept of hyperenergetic graphs from our viewpoint.

\subsection*{Hyperenergetic graphs}

Recall that a graph is called hyperenergetic if $\mathcal{E}(G)>2n-2$.   This concept was considered in the literature because of the conjecture that the complete graph was the graph with largest energy, which was quickly shown to be false.

  In the case of energy of a vertex we may define a similar concept by defining a vertex to be hyperenergetic if it has larger energy of the vertex of a complete graph. However, if one wants to compare graphs with different sizes we would like that the definition wouldn't depend on $n$, for this reason
we will be interested in vertices such that $\mathcal{E}_{G}(v)\geq2$.  It is a priori not clear that there are graphs such that \textbf{all} the vertices are larger than $2$, we will see that there are infinite families of them.

\begin{defi} Let $G$ be a graph and $v\in G$.
\begin{enumerate}
\item  The vertex $v$ is called hyperenergetic if $\mathcal{E}_{G}(v)\geq 2$.
\item$G$ is called completely hyperenergetic if  $\mathcal{E}_{G}(v)\geq 2$ for all $v$.
\item $G$ is  is called completely non-hyperenergetic if  $\mathcal{E}_{G}(v)< 2$ for all $v$.
\end{enumerate}
\end{defi}

A first observation is that for a vertex to be hyperenergetic then the degree must be larger or equal $4$.

\begin{prop}
For a vertex $v$ if  $deg(v)\leq4$ then
$$\mathcal{E}_G(v)\leq 2,$$
with equality if and only if $G=S_4$ and is its center.
\end{prop}
\begin{proof}
Follows from Theorem \ref{McClelland vertex}.
\end{proof}

Thus completely hyperenergetic graphs should have degrees larger than $4$. In particular we obtain the following trivial corollary that we state for comparison.

\begin{prop}
A $d$-regular graph is completely non-hyperenergetic whenever $d\leq4.$
\end{prop}

On the other hand if $d$ is large, in the absence quadrangles we can ensure hyperenergy.

 \begin{prop}
A quadrangle free $d$-regular graph is completely hyperenergetic if $d\geq8.$
\end{prop}

\begin{proof}
Since $G$ is quadrangle free, by Corollary \ref{no quadrangles }
 then for all vertex $v\in V$
$$ \mathcal{E}_G(v)\geq \sqrt{d} \sqrt{\frac{d}{2d-1}}>2.065$$
\end{proof}

The proof  can be easily generalized to show that graphs such that  $\delta$ is large enough compared with $\Delta$ are also hypereneregetic as long as they do not have quadrangles.
\begin{prop}
A quadrangle free graph is completely hyperenergetic if $\delta\geq 2\sqrt{\Delta}+2.$
\end{prop}

\subsection*{Hypoenergetic graphs}

Similarly to hypernegetic graphs we con consider hypoenergetic graphs. Recall that hypoenergetic graphs are graphs that have small energy. By this we mean that the energy is strictly less than $n$. The same ideas as in the previous sections motivate the following definitions.
\begin{defi}
Let $G$ be a graph and $v\in G$.
\begin{enumerate} \item $v$ is called hypoenergetic if $\mathcal{E}_G(v)<1$.
\item $G$ is called completely hypoenergetic if  $\mathcal{E}_G(v)<1$ for all $v$.
\item $G$ is called completely non-hyperenergetic if  $\mathcal{E}_G(v)\geq 1$ for all $v$.
\end{enumerate}
\end{defi}

We first note that pendant vertices are in general hypoenergetic.  More precisely from Theorem {McClelland vertex} we deduce the following.

\begin{prop}
Let $G\neq K_2$. If $v\in G$  and $deg(v)=1$ then $\mathcal{E}(G)<1$.
\end{prop}

In particular, since trees have always pendant vertices, they are never completely non-hypoenergetic.  On the other hand, for non-trivial graphs not all vertices can be hypoenergetic.

\begin{prop}
The only completely hypoenergetic graphs are union of isolated vertices.
\end{prop}
\begin{proof}If  has an edge then $\Delta\neq0$ and by Theorem \ref{4.1}, for any vertex $v_i$ we have $\mathcal{E}_G(v_i)\geq d_i/\Delta$. Taking $v_i$ with $d_i=\Delta$ we have $\mathcal{E}_G(v)\geq1.$
\end{proof}

\begin{prop}
Any $d$-regular graph is completely non-hypoenergetic, for $d\geq1$. That is,
\begin{equation}\label{below} \mathcal{E}_G(v_i)\geq 1\text{ for all $i$.}\end{equation}
\end{prop}
\begin{proof}
The bound from Theorem \ref{4.1} is in this case $d_i/\Delta=d/d=1.$
\end{proof}

\begin{rem} There are regular graphs with equality in  $\eqref{below}$ for all $i$ , for an example consider the graph $K_{n,n}$ which is vertex transitive and has rank $2$. Its nonzero eigenvalues are $n$ and $-n$, each with multiplicity $1$. We conclude the total energy of the graph is $2n$ and since the graph is vertex transitive each vertex has energy $1$.
\end{rem}

Now we use Corollary \ref{corbipartite} to obtain the following result, an improvement of McClelland's inequality, which is specially useful when one of the parts in a bipartite graph is very small.

\begin{prop}
Let $G$ be a bipartite graph connected graph with parts of sizes $n_1$ and $n_2$. Then
\begin{equation} \label{eqbi}  \frac{2 (n_2-1+\sqrt{m-n_2+1})}{\sqrt{\Delta}}\leq\mathcal{E}(G)\leq 2 \sqrt{n_1 \cdot m  } \end{equation}
\end{prop}
\begin{proof} By Cauchy-Schwarz  inequality we have $$ \sum_{v \in V_1} \sqrt{ deg(v)}<\sqrt {|V_1| \sum_{v \in V_1} deg(v)}=\sqrt{n_1 m}.$$
For the lower bound we use the fact that $\sqrt{a+1}-\sqrt{a}\geq\sqrt{b}-\sqrt{b-1}$ if  $a\geq b$ and then the best possible configuration is when one of the vertices has the largest degree and the rest have degree $1$.
\end{proof}

As a corollary since for bipartite graphs we have  $m\leq n_1n_2$ we get the following examples of hypoenergetic graphs.

\begin{prop}
Let G be a bipartite graph with parts of size $n_1$ and $n_2$ then if $2n_1\leq \sqrt{n_2}$ then
$$ \mathcal{E}(G)<n.$$
\end{prop}
\begin{proof} If $m=n_1n_2$ we have a complete bipartite graph for which the result follows. Now, suppose that $m<n_1n_2$, then $$2\sqrt{m\cdot n_1 }\leq 2n_1\sqrt{n_2}\leq n_2<n_1+n_2=n.$$
\end{proof}

We may specialize for trees to assume a weaker condition for the relation between the sizes of the parts.

\begin{prop}
Let $T$ be a bipartite graph with no cycles (i.e. a tree) and parts of size $n_1$ and $n_2$ if $n_1\leq n/4$ then
$$ \mathcal{E}(G)<n$$
\end{prop}
\begin{proof}In this case $m=n_1+n_2-1$. Thus the inequality (\ref{eqbi}) gives $$2\sqrt{ (n_1+n_2-1) n_1}<2\sqrt{n\cdot n/4}=n.$$
\end{proof}

Finally,  together with edge deletion properties of energies we may also consider graphs with large independent set, not necessarily bipartite.
\begin{prop}
Let $G=(V,E)$ be a graph of size $n_1+n_2$ such that $V=W_1\cup W_2$ with $W_1$ and $W_2$ of size $n_1$ and $n_2$, respectively, and such that no vertices of $W_2$ are connected between them.  If $n_1\leq .4 \sqrt{n_2}$ then
$$ \mathcal{E}(G)<n.$$
\end{prop}
\begin{proof}From Corollary 4.7 in \cite{LiShiGu}  we know that if $e$ is an edge of a graph, then $|\mathcal{E}(G)- \mathcal{E}(G-\{e\})|\leq2. $

By (at most $n_1(n_1-1)/2$) repeated edge deletions, we see that the energy of $G$ is bounded by $\mathcal{E}(\tilde G)+n_1^2$,  where $\tilde G$ is a bipartite graph with sizes $n_1$ and $n_2$ , since $n_1\leq .4 \sqrt{n_2}$, we have $\mathcal{E}(\tilde G)+n_1^2\leq \mathcal{E}(\tilde G)+.16 n_2 $,. Now, for  $\tilde G$ we may use (\ref{eqbi}) to obtain $$\mathcal{E}(\tilde G)\leq 2\sqrt{m\cdot n_1}\leq 2n_1\sqrt{n_2}\leq .8 n_2. $$
Thus $\mathcal{E}(G)\leq(.16+.8)n_2=n_2<n.$
\end{proof}

\section{Examples and Counterexamples}\label{ExaCounter}

In this section, firstly, we give precise formulas for the energy of the vertices of certain important graphs used commonly in the literature and, secondly, we explain examples and counterexamples of natural questions regarding the energy of a vertex.
\subsection{Some Transitive graphs}

It is straight-forward that if there is an  automorphism $f:G\to G$ such that $f(v)=w$ then the energy of $v$ equals the energy of $w$.

Now, recall that a simple graph $G$ is called vertex-transitive if for every pair of vertices $v$ and $w$ there is an automorphism sending $v$ to $w$. Some examples of transitive graphs are given by the hypercubes, the cycles and the complete graphs. The energy of vertices of transitive graphs can be calculated very simply from the total energy as follows:
\begin{equation}\label{Equa34}
    \mathcal{E}_G(v_i)=\frac{\mathcal{E}(G)}{n}, \quad\quad~~~\text{for }   i=1,\ldots,n.
\end{equation}

By using (\ref{Equa34}) and the expression of the energy for the complete Graph, the cycle and the hypercube given in Chen et. al. \cite{CX} and Li et. al. \cite{LiShiGu}, we have the following results.

\subsection*{Complete graph $K_{n}$}

The complete graph on $n$ vertices, denoted by $K_n$, is the graph in which there is an edge between any two vertices.
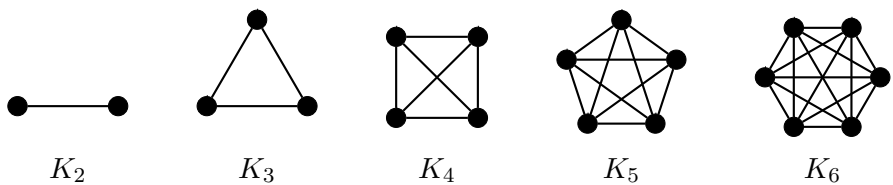
\begin{figure}[h]
\begin{tikzpicture}
[mystyle/.style={scale=0.7, draw,shape=circle,fill=black}]
\def\ngon{3}
\node[regular polygon,regular polygon sides=\ngon,minimum size=1.5cm] (p) {};
\node[mystyle] (p1) at (p.corner 3){};
\node[mystyle] (p3) at (p.corner 2){};
 \draw[thick] (p1) -- (p3);
 \node [label=below:$K_2$] (*) at (0,-0. 8 ) {};
\end{tikzpicture} \qquad
\begin{tikzpicture}[mystyle/.style={scale=0.7,draw,shape=circle,fill=black}]
\def\ngon{3}
\node[regular polygon,regular polygon sides=\ngon,minimum size=1.5cm] (p) {};
\foreach\x in {1,...,\ngon}{\node[mystyle] (p\x) at (p.corner \x){};}
\foreach\x in {1,...,\numexpr\ngon-1\relax}{
  \foreach\y in {\x,...,\ngon}{
    \draw[thick] (p\x) -- (p\y);
  }
}
 \node [label=below:$K_\ngon$] (*) at (0,-0. 8  ) {};
\end{tikzpicture}
  \qquad
\begin{tikzpicture}[mystyle/.style={scale=0.7,draw,shape=circle,fill=black}]
\def\ngon{4}
\node[regular polygon,regular polygon sides=\ngon,minimum size=1.5cm] (p) {};
\foreach\x in {1,...,\ngon}{\node[mystyle] (p\x) at (p.corner \x){};}
\foreach\x in {1,...,\numexpr\ngon-1\relax}{
  \foreach\y in {\x,...,\ngon}{
    \draw[thick] (p\x) -- (p\y);
  }
}
 \node [label=below:$K_\ngon$] (*) at (0,-0. 8  ) {};
\end{tikzpicture}
  \qquad
\begin{tikzpicture}
[mystyle/.style={scale=0.7, draw,shape=circle,fill=black}]
\def\ngon{5}
\node[regular polygon,regular polygon sides=\ngon,minimum size=1.5cm] (p) {};
\foreach\x in {1,...,\ngon}{\node[mystyle] (p\x) at (p.corner \x){};}
\foreach\x in {1,...,\numexpr\ngon-1\relax}{
  \foreach\y in {\x,...,\ngon}{
    \draw[thick] (p\x) -- (p\y);
  }
}
 \node [label=below:$K_\ngon$] (*) at (0,-0. 8  ) {};
\end{tikzpicture}
  \qquad
\begin{tikzpicture}
[mystyle/.style={scale=0.7, draw,shape=circle,fill=black}]
\def\ngon{6}
\node[regular polygon,regular polygon sides=\ngon,minimum size=1.5cm] (p) {};
\foreach\x in {1,...,\ngon}{\node[mystyle] (p\x) at (p.corner \x){};}
\foreach\x in {1,...,\numexpr\ngon-1\relax}{
  \foreach\y in {\x,...,\ngon}{
   \draw[thick] (p\x) -- (p\y);
  }
}
 \node [label=below:$K_\ngon$] (*) at (0,-0. 8  ) {};
\end{tikzpicture}
 \caption{Complete graphs $K_2$, $K_3$, $K_4$, $K_5$ and $K_6$ }

\end{figure}

Its spectrum is $\{-1^{(n-1)},n-1^{1}\}$ and from this the energy is $2n-2$. From here we deduce the energy of the vertices is given by

\begin{equation}
  \mathcal{E}_{K_{n}}(v_i)=\cfrac{2(n-1)}{n}, \quad\quad~~~\text{for } i=1,\dots,n.
\end{equation}

\subsection*{Hypercube $Q_n$}

The $n$-dimensional hypercube, denoted by $Q_n$, is the graph whose vertices are the $n$-tuples with entries in $\{0,1\}$ and whose edges are the pairs of $n$-tuples that differ in exactly one position.

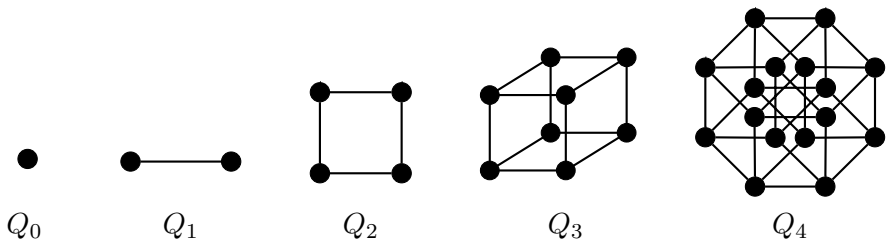
\begin{figure}[h]
\centering
\begin{tikzpicture}
[mystyle/.style={scale=0.7, draw,shape=circle,fill=black}]
\def\ngon{3}
\node[mystyle] (p1) at (p.corner 3){};
 \node [label=below:$Q_0$] (*) at (-0.8,-0.45  ) {};
\end{tikzpicture}
  \qquad
\begin{tikzpicture}
[mystyle/.style={scale=0.7, draw,shape=circle,fill=black}]
\def\ngon{3}
\node[regular polygon,regular polygon sides=\ngon,minimum size=1.5cm] (p) {};
\node[mystyle] (p1) at (p.corner 3){};
\node[mystyle] (p3) at (p.corner 2){};

 \draw[thick] (p1) -- (p3);
 \node [label=below:$Q_1$] (*) at (0,-0. 8  ) {};
\end{tikzpicture}
  \qquad\begin{tikzpicture}
[mystyle/.style={scale=0.7, draw,shape=circle,fill=black}]
\def\ngon{4}
\node[regular polygon,regular polygon sides=\ngon,minimum size=1.5cm] (p) {};
\foreach\x in {1,...,\ngon}{\node[mystyle] (p\x) at (p.corner \x){};}
\foreach\x in {1,...,\numexpr\ngon-1\relax}{
  \foreach\y in {\x,...,\numexpr\x+1\relax}{
    \draw[thick] (p\x) -- (p\y);
  }
}
 \draw[thick] (p1) -- (p\ngon);
  \node [label=below:$Q_2$] (*) at (0,-0. 8  ) {};
 \end{tikzpicture}
\qquad
\begin{tikzpicture}[scale=1]
\filldraw [black]
(-1,-0.5) circle (3.5 pt)
(-1,0.5) circle (3.5 pt)
(0,0.5) circle (3.5 pt)
(0,-0.5) circle (3.5 pt)

(-0.2,0) circle (3.5 pt)
(-0.2,1) circle (3.5 pt)
(0.8,1) circle (3.5 pt)
(0.8,0) circle (3.5 pt);
\draw[thick] (-1,-0.5) -- (-0.2,0);
\draw[thick] (-1,0.5) -- (-0.2,1);
\draw[thick] (0,0.5) -- (0.8,1);
\draw[thick] (0,-0.5) -- (0.8,0);

\draw[thick] (-1,-0.5) --  (-1,0.5);
\draw[thick] (-0.2,0) -- (-0.2,1);
\draw[thick] (0,0.5) --(0,-0.5);
\draw[thick]  (0.8,1) -- (0.8,0);

\draw[thick] (-1,-0.5) -- (0,-0.5) ;
\draw[thick] (-0.2,0) -- (0.8,0);
\draw[thick] (0,0.5) --(-1,0.5);
\draw[thick]  (0.8,1) --(-0.2,1) ;
 \node [label=below:$Q_3$] (*) at (0,-0. 8  ) {};
\end{tikzpicture}\qquad
  \begin{tikzpicture}
[mystyle/.style={scale=0.7, draw,shape=circle,fill=black}]
\def\ngon{8}
\node[regular polygon,regular polygon sides=\ngon,minimum size=2.4cm] (p) {};

\node[regular polygon,regular polygon sides=\ngon,minimum size=1cm] (d) {};
\foreach\x in {1,...,\ngon}{\node[mystyle] (p\x) at (p.corner \x){};}
\foreach\x in {1,...,\ngon}{\node[mystyle] (d\x) at (d.corner \x){};}

\foreach\x in {1,...,\numexpr\ngon-1\relax}{
  \foreach\y in {\x,...,\numexpr\x+1\relax}{
    \draw[thick] (p\x) -- (p\y);

  }
}

\draw[thick] (p1) -- (d2)-- (p3) -- (d4)-- (p5) -- (d6)-- (p7) -- (d8)-- (p1);
\draw[thick] (d1) -- (p2)-- (d3) -- (p4)-- (d5) -- (p6)-- (d7) -- (p8)-- (d1);
\draw[thick] (d1) -- (d4)-- (d7) -- (d2)-- (d5) -- (d8)-- (d3) -- (d6)-- (d1);

\draw[thick] (p1) -- (p\ngon);

 \node [label=below:$Q_4$] (*) at (0,-1.2  ) {};
 \end{tikzpicture}

 \caption{Hypercube of dimension 0, 1, 2, 3 and 4.}
\end{figure}

The eigenvalues are easy to calculate since this is the cartesian product $K_2\times K_2\cdots\times K_2$ and then they follow the binomial pattern $\{(n-2k)^{\binom{n}{k}}\}$. The energy is easily calculated by the formula $2\lceil \frac{n}{2}\rceil \displaystyle\binom{n}{\lceil \frac{n}{2}\rceil}$ and thus the vertex energies are given as follows.

\begin{equation}
  \mathcal{E}_{Q_{n}}(v_i)=\cfrac{2\lceil \frac{n}{2}\rceil \displaystyle\binom{n}{\lceil \frac{n}{2}\rceil}}{n}, \quad\quad~~~\text{for } i=1,\dots,n.
\end{equation}

\subsection*{Cycle $C_n$}

The cycle $C_n$ is the graph with vertex set $V=[n]$  and edge set given by $E=\{(v_i,v_{i+1})\}_{1\leq i<n}\cup(v_0,v_n)$.

\begin{figure}[h]
\begin{tikzpicture}
[mystyle/.style={scale=0.7, draw,shape=circle,fill=black}]
\def\ngon{3}
\node[regular polygon,regular polygon sides=\ngon,minimum size=1.5cm] (p) {};
\node[mystyle] (p1) at (p.corner 3){};
\node[mystyle] (p3) at (p.corner 2){};

 \draw[thick] (p1) -- (p3);
  \node [label=below:$C_2$] (*) at (0,-0. 8  ) {};
\end{tikzpicture}
  \qquad\begin{tikzpicture}
[mystyle/.style={scale=0.7, draw,shape=circle,fill=black}]
\def\ngon{3}
\node[regular polygon,regular polygon sides=\ngon,minimum size=1.5cm] (p) {};
\foreach\x in {1,...,\ngon}{\node[mystyle] (p\x) at (p.corner \x){};}
\foreach\x in {1,...,\numexpr\ngon-1\relax}{
  \foreach\y in {\x,...,\numexpr\x+1\relax}{
    \draw[thick] (p\x) -- (p\y);
  }
}
 \draw[thick] (p1) -- (p\ngon);
   \node [label=below:$C_\ngon$] (*) at (0,-0. 7  ) {};
\end{tikzpicture}
  \qquad\begin{tikzpicture}
[mystyle/.style={scale=0.7, draw,shape=circle,fill=black}]
\def\ngon{4}
\node[regular polygon,regular polygon sides=\ngon,minimum size=1.5cm] (p) {};
\foreach\x in {1,...,\ngon}{\node[mystyle] (p\x) at (p.corner \x){};}
\foreach\x in {1,...,\numexpr\ngon-1\relax}{
  \foreach\y in {\x,...,\numexpr\x+1\relax}{
    \draw[thick] (p\x) -- (p\y);
  }
}
 \draw[thick] (p1) -- (p\ngon);
    \node [label=below:$C_\ngon$] (*) at (0,-0. 8  ) {};
\end{tikzpicture}
  \qquad\begin{tikzpicture}
[mystyle/.style={scale=0.7, draw,shape=circle,fill=black}]
\def\ngon{5}
\node[regular polygon,regular polygon sides=\ngon,minimum size=1.5cm] (p) {};
\foreach\x in {1,...,\ngon}{\node[mystyle] (p\x) at (p.corner \x){};}
\foreach\x in {1,...,\numexpr\ngon-1\relax}{
  \foreach\y in {\x,...,\numexpr\x+1\relax}{
    \draw[thick] (p\x) -- (p\y);
  }
}
 \draw[thick] (p1) -- (p\ngon);
    \node [label=below:$C_\ngon$] (*) at (0,-0. 8  ) {};
\end{tikzpicture}
  \qquad\begin{tikzpicture}
[mystyle/.style={scale=0.7, draw,shape=circle,fill=black}]
\def\ngon{6}
\node[regular polygon,regular polygon sides=\ngon,minimum size=1.5cm] (p) {};
\foreach\x in {1,...,\ngon}{\node[mystyle] (p\x) at (p.corner \x){};}
\foreach\x in {1,...,\numexpr\ngon-1\relax}{
  \foreach\y in {\x,...,\numexpr\x+1\relax}{
    \draw[thick] (p\x) -- (p\y);
  }
}
 \draw[thick] (p1) -- (p\ngon);
    \node [label=below:$C_\ngon$] (*) at (0,-0. 8  ) {};
\end{tikzpicture}
  \qquad

 \caption{Cycle graphs $C_2$, $C_3$, $C_4$, $C_5$ and $C_6$ }

\end{figure}
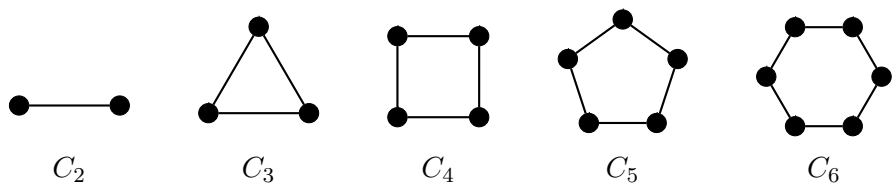

 Its spectrum is $\{2\cos(2k\pi/n) | k\in \{1,2,\dots,n)\}$ and then the vertex energies are given by
\begin{equation}
\mathcal{E}_{C_{n}}(v_i)=
\begin{cases}
\cfrac{4\cos\frac{\pi}{n}}{n\sin\frac{\pi}{n}}, & \text{if $n \equiv 0 \Mod 4$}\\
\cfrac{4}{n\sin\frac{\pi}{n}}, & \text{if $n \equiv 2 \Mod 4$}\\
\cfrac{2}{n\sin\frac{\pi}{2n}}, & \text{if $n \equiv 1 \Mod 2$}
\end{cases} \quad\quad~~~\text{for } i=1,\dots,n.
\end{equation}
Note that as $n$ tends to infinity in all of the cases we have that  \begin{equation}\label{limit1} \lim_{n\to\infty}\mathcal{E}_{C_{n}}(v_i)= \frac{4}{\pi}.\end{equation}

\subsection{Some graphs with many symmetries}

Apart from transitive graphs there are some graphs with large symmetries for which the energy of vertices can be easily calculated from the total energy of the graph. Edge transitive graphs are such an example. A particular example being the bipartite complete graphs.

A graph $G$ is said to be edge transitive if for each $\{v_i,v_j\},\{v_k,v_l\} \in E$ there exists an automorphism $\varphi$ of $G$ such that $\{\varphi(v_i),\varphi(v_j)\}=\{v_k,v_l\}$.

Edge transitive graphs with no isolated vertices can have at most two distinct vertex energies. To see this pick a fixed edge $\{v_i,v_j\}$ and for every other vertex $v_k$ pick an edge $\{v_k,v_l\}$ and an automorphism $\varphi$ with $\{\varphi(v_k),\varphi(v_l)\}=\{v_i,v_j\}$. it follows $\varphi(v_k)=v_i$ or $v_j$ and so the energy of $v_k$ equals that of $v_i$ or $v_j$.

\subsection*{Complete Bipartite Graphs}
The complete bipartite graph $K_{n_1,n_2}$ is a graph where the set of vertices is divided into two parts $V_1$ and $V_2$ and any two vertices $v_1\in V_1$ and $v_2\in V_2$ are connected by and edge. This graph is usually called an $n_1,n_2$- complete bipartite graph. The case $n_1=1$ and $n_2=n-1$ is called the star $S_{n}$.

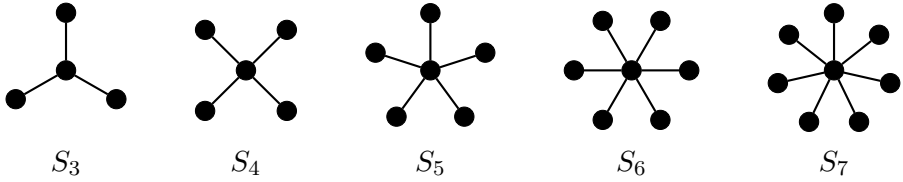
\begin{figure}[h]
\begin{tikzpicture}
[mystyle/.style={scale=0.7, draw,shape=circle,fill=black}]
\def\ngon{3}
\node[regular polygon,regular polygon sides=\ngon,minimum size=1.5cm] (p) {};
\foreach\x in {1,...,\ngon}{\node[mystyle] (p\x) at (p.corner \x){};}
\node[mystyle] (p0) at (0,0) {};
\foreach\x in {1,...,\ngon}
{
 \draw[thick] (p0) -- (p\x);
}
  \node [label=below:$S_{\ngon}$] (*) at (0,-0.8) {};
 \end{tikzpicture}
  \qquad
\begin{tikzpicture}
[mystyle/.style={scale=0.7, draw,shape=circle,fill=black}]
\def\ngon{4}
\node[regular polygon,regular polygon sides=\ngon,minimum size=1.5cm] (p) {};
\foreach\x in {1,...,\ngon}{\node[mystyle] (p\x) at (p.corner \x){};}
\node[mystyle] (p0) at (0,0) {};
\foreach\x in {1,...,\ngon}
{
 \draw[thick] (p0) -- (p\x);
}
  \node [label=below:$S_{\ngon}$] (*) at (0,-0.8) {};
 \end{tikzpicture}
  \qquad
\begin{tikzpicture}
[mystyle/.style={scale=0.7, draw,shape=circle,fill=black}]
\def\ngon{5}
\node[regular polygon,regular polygon sides=\ngon,minimum size=1.5cm] (p) {};
\foreach\x in {1,...,\ngon}{\node[mystyle] (p\x) at (p.corner \x){};}
\node[mystyle] (p0) at (0,0) {};
\foreach\x in {1,...,\ngon}
{
 \draw[thick] (p0) -- (p\x);
}
  \node [label=below:$S_{\ngon}$] (*) at (0,-0.8) {};
 \end{tikzpicture}
  \qquad
\begin{tikzpicture}
[mystyle/.style={scale=0.7, draw,shape=circle,fill=black}]
\def\ngon{6}
\node[regular polygon,regular polygon sides=\ngon,minimum size=1.5cm] (p) {};
\foreach\x in {1,...,\ngon}{\node[mystyle] (p\x) at (p.corner \x){};}
\node[mystyle] (p0) at (0,0) {};
\foreach\x in {1,...,\ngon}
{
 \draw[thick] (p0) -- (p\x);
}
  \node [label=below:$S_{\ngon}$] (*) at (0,-0.8) {};
 \end{tikzpicture}
  \qquad
  \begin{tikzpicture}
[mystyle/.style={scale=0.7, draw,shape=circle,fill=black}]
\def\ngon{7}
\node[regular polygon,regular polygon sides=\ngon,minimum size=1.5cm] (p) {};
\foreach\x in {1,...,\ngon}{\node[mystyle] (p\x) at (p.corner \x){};}
\node[mystyle] (p0) at (0,0) {};
\foreach\x in {1,...,\ngon}
{
 \draw[thick] (p0) -- (p\x);
}
  \node [label=below:$S_{\ngon}$] (*) at (0,-0.8) {};
 \end{tikzpicture}
  \caption{Stars $S_1$, $S_2$, $S_3$, $S_4$ and $S_5$ }

\end{figure}

 The spectrum of $K_{n_1,n_2}$ is known to be $\{\sqrt{n_1n_2},-\sqrt{n_1n_2},0^{(n_1+n_2-2)}\}$ and thus the energy of $K_{n_1,n_2}$ is given by $2\sqrt{n_1n_2}$.

From Proposition \ref{bipartite} we see that each of the parts contributes in $\sqrt{n_1n_2}$. On the other hand since $K_{n,m}$ is edge transitive, all the vertices in each part have the same energy. Thus,  we have that
\begin{equation}
\mathcal{E}_{K_{n_1,n_2}}(v) =
\begin{cases}
      \cfrac{\sqrt{n_2}}{\sqrt{n_1}}, & \text{if $v \in V_1$,}\\

      \cfrac{\sqrt{n_1}}{\sqrt{n_2}}, & \text{if $v \in V_2$.}
\end{cases}
\end{equation}

In particular, for a star graph, $S_{n}$, we recover the fact that
\begin{equation}\label{Star}
\mathcal{E}_{S_{n}}(v_j) =
\begin{cases}
      \sqrt{n-1}, & \text{if $v_j$ is the center of the star $S_{n}$,}\\

      \cfrac{1}{\sqrt{n-1}}, & \text{otherwise.}
\end{cases}
\end{equation}

\subsection*{The friendship graph}

The friendship graph $F_n$ is the graph  with $2n+1$ vertices $\{v_1,...,v_{2n+1}\}$ in which $v_1$ is connected to every other vertex and the only other edges are $\{v_{2i},v_{2i+1}\}$ for $1\leq i \leq n$.

\begin{figure}[h]
\begin{tikzpicture}
[mystyle/.style={scale=0.7, draw,shape=circle,fill=black}]
\def\ngon{3}
\node[regular polygon,regular polygon sides=\ngon,minimum size=1.5cm] (p) {};
\foreach\x in {1,...,\ngon}{\node[mystyle] (p\x) at (p.corner \x){};}
\foreach\x in {1,...,\numexpr\ngon-1\relax}{
  \foreach\y in {\x,...,\numexpr\x+1\relax}{
    \draw[thick] (p\x) -- (p\y);
  }
}
 \draw[thick] (p1) -- (p\ngon);
 \node [label=below:$F_{1}$] (*) at (0,-0.8) {};
\end{tikzpicture}
\qquad
\begin{tikzpicture}
[mystyle/.style={scale=0.7, draw,shape=circle,fill=black}]
\def\ngon{4}
\node[regular polygon,regular polygon sides=\ngon,minimum size=1.5cm] (p) {};
\foreach\x in {1,...,\ngon}{\node[mystyle] (p\x) at (p.corner \x){};}
\foreach\x in {1,3}{
  \foreach\y in {\x,...,\numexpr\x+1\relax}{
    \draw[thick] (p\x) -- (p\y);
  }
}
\node[mystyle] (p0) at (0,0) {};
\foreach\x in {1,...,\ngon}
{
 \draw[thick] (p0) -- (p\x);
}
 \node [label=below:$F_{2}$] (*) at (0,-0.8) {};
 \end{tikzpicture}
 \qquad
\begin{tikzpicture}
[mystyle/.style={scale=0.7, draw,shape=circle,fill=black}]
\def\ngon{6}
\node[regular polygon,regular polygon sides=\ngon,minimum size=1.5cm] (p) {};
\foreach\x in {1,...,\ngon}{\node[mystyle] (p\x) at (p.corner \x){};}
\foreach\x in {1,3,5}{
  \foreach\y in {\x,...,\numexpr\x+1\relax}{
    \draw[thick] (p\x) -- (p\y);
  }
}
\node[mystyle] (p0) at (0,0) {};
\foreach\x in {1,...,\ngon}
{
 \draw[thick] (p0) -- (p\x);
}
 \node [label=below:$F_{3}$] (*) at (0,-0.8) {};
  \end{tikzpicture}
 \qquad
 \begin{tikzpicture}
[mystyle/.style={scale=0.7, draw,shape=circle,fill=black}]
\def\ngon{8}
\node[regular polygon,regular polygon sides=\ngon,minimum size=1.5cm] (p) {};
\foreach\x in {1,...,\ngon}{\node[mystyle] (p\x) at (p.corner \x){};}
\foreach\x in {1,3,5,7}{
  \foreach\y in {\x,...,\numexpr\x+1\relax}{
    \draw[thick] (p\x) -- (p\y);
  }
}
\node[mystyle] (p0) at (0,0) {};
\foreach\x in {1,...,\ngon}
{
 \draw[thick] (p0) -- (p\x);
}
 \node [label=below:$F_{4}$] (*) at (0,-0.8) {}; \end{tikzpicture}
  \qquad
 \begin{tikzpicture}
[mystyle/.style={scale=0.7, draw,shape=circle,fill=black}]
\def\ngon{10}
\node[regular polygon,regular polygon sides=\ngon,minimum size=1.5cm] (p) {};
\foreach\x in {1,...,\ngon}{\node[mystyle] (p\x) at (p.corner \x){};}
\foreach\x in {1,3,5,7,9}{
  \foreach\y in {\x,...,\numexpr\x+1\relax}{
    \draw[thick] (p\x) -- (p\y);
  }
}
\node[mystyle] (p0) at (0,0) {};
\foreach\x in {1,...,\ngon}
{
 \draw[thick] (p0) -- (p\x);
}
 \node [label=below:$F_{5}$] (*) at (0,-0.8) {};
 \end{tikzpicture}
 \caption{Friendship graphs $F_1$, $F_2$, $F_3$, $F_4$ and $F_5$ }
\end{figure}
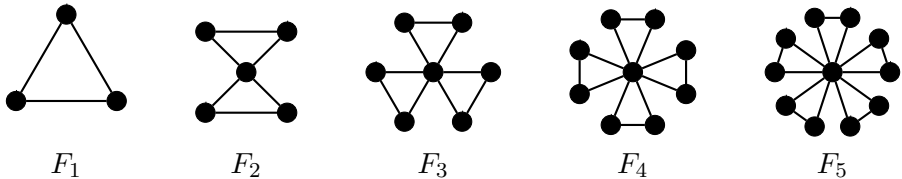

 While this is not an edge transitive graph it is clear that all the vertices $v_i$, $i>1$, should have the same energy.

The spectra of the graph is
$\{[\frac{1}{2}-\frac{1}{2}\sqrt{1+8n}]^1,[-1]^n,[1]^{n-1},[\frac{1}{2}+\frac{1}{2}\sqrt{1+8n}]^1\}$ and thus the total energy is given by $2n-1+ \sqrt{1+8n}.$
Since there are only $4$ distinct eigenvalues, the weights $p_{1i}$ can be calculated with a $4$ by $4$ system of linear equations coming from calculating the first moments, on one hand, directly counting walks and on the other, from the formula \ref{momentos}
\begin{eqnarray*}p_{11}\lambda_1^0+p_{12}\lambda_2^0+p_{13}\lambda_3^0+p_{14}\lambda_4^0&=&1\\
p_{11}\lambda_1+p_{12}\lambda_2+p_{13}\lambda_3+p_{14}\lambda_4&=&0\\
p_{11}\lambda^2_1+p_{12}\lambda^2_2+p_{13}\lambda_3^2+p_{14}\lambda_4^2&=&2n \\
p_{11}\lambda^3_1+p_{12}\lambda^3_2+p_{13}\lambda_3^3+p_{14}\lambda_4^3&=&2n
\end{eqnarray*}
 This gives the weights of the eigenvalues corresponding to $v_1$ and thus the energy of $v_1$ is given by the sum of the product of each weight with the absolute value of each corresponding eigenvalue.  Finally for the rest of the vertices ($v_i$ for $i>1$) , the energy is simply calculating by dividing the remaining energy, $\big(\mathcal E(F_n)-\mathcal{E}_{F_n}(v_1)\big)$, equally among the $2n$ vertices.  Thus  the energies are given by

\begin{equation}
\mathcal{E}_{F_{n}}(v_i)=
\begin{cases}
\cfrac{4n}{\sqrt{8n+1}}, & \text{if $i=1$},\\ \\
\cfrac{1+4n+\left( 2n-1\right)  \sqrt{8n+1}}{2n \sqrt{8n+1}}, & \text{if $i\neq 1$}.
\end{cases}
\end{equation}

\subsection{The path}

We end with the path $P_n$ with vertices $\{1,\cdots,n+1\}$ . The energy of its vertices can be calculated Lemma \ref{W1}.
The eigenvalues and eigenvectors are given by
\begin{equation}
  \lambda_j =2\cos\left(\cfrac{j\pi}{n+1}\right),\qquad 1 \le j \le n,
\end{equation}
and
\begin{equation}
    u_{ij} = \frac{\sqrt{2}\sin(ij\pi/(n+1))}{\sqrt{n+1}},\qquad 1\le i,j \le n.
\end{equation}
From (\ref{Equa35}) we have that the energy of the vertex $v_i$ is given by
\begin{equation}
\mathcal{E}_{P_n}(v_i)=\frac{4}{n+1}\sum^{n}_{j=1} |\cos(\frac{j\pi}{n+1})| \sin\left(\frac{ij\pi}{n+1}\right)^2
\end{equation}

Notice that this energy is exactly the $(n+1)$-th approximation of the Riemann integral $ \int^1_0 4| \cos(\pi x)| \sin\left(i \pi x\right)^2 dx $, from where, taking the limit as $n$ tends to infinity we see that
 \begin{equation} \label{limit2}\lim_{n\to\infty} \mathcal{E}_{P_n}(i)= \frac{4}{\pi}+\frac{4(-1)^i}{\pi(4i^2-1)}.\end{equation}
Also note that if $i\gg1$, then
\begin{equation} \label{limit3}\lim_{n\to\infty} \mathcal{E}_{P_n}(i)\sim \frac{4}{\pi},\end{equation}
which was the limit found for the cycle. This is no coincidence as we will explain in Section \ref{LFG}.

\subsection{Counterexamples}

Many natural questions and conjectures may be made about graph energy, for example: Is the energy of  a vertex in a graph $G$ always larger than the energy of the same vertex in any induced subgraph of $G$ containing $v$? Does the energy of a vertex increase when we attach an edge to it? If the vertex energies of two graphs are the same, are they isomorphic? In this section we provide counterexamples to various questions of this kind.

\begin{exa}

It is a result in graph theory that if $G$ is a graph and $H$ is an induced subgraph of $G$ then $\epsilon(H)\leq \epsilon(G)$. One can ask if the same can be said about $\epsilon_H(v_i)$ and $\epsilon_G(v_i)$ for each vertex $v_i\in H$. The following graphs serve as counterexamples:

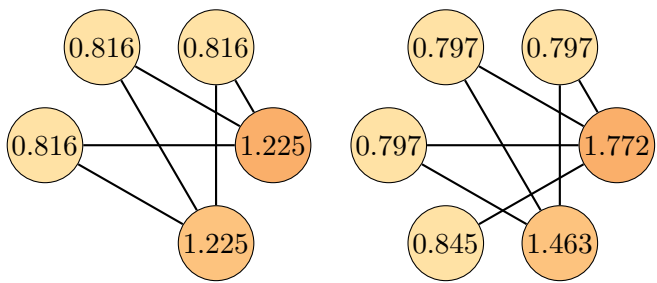
\begin{center}
\begin{tikzpicture}
\node[style={circle,fill=yellow!75!red!40,draw,minimum size=1em,inner sep=1pt]}] (1) at (1*360/6:1.5) {$0.816$};
\node[style={circle,fill=yellow!75!red!40,draw,minimum size=1em,inner sep=1pt]}] (2) at (2*360/6:1.5) {$0.816$};
\node[style={circle,fill=yellow!75!red!40,draw,minimum size=1em,inner sep=1pt]}] (3) at (3*360/6:1.5) {$0.816$};
\node[style={circle,fill=yellow!54!red!56,draw,minimum size=1em,inner sep=1pt]}] (5) at (5*360/6:1.5) {$1.225$};
\node[style={circle,fill=yellow!44!red!64,draw,minimum size=1em,inner sep=1pt]}] (6) at (6*360/6:1.5) {$1.225$};
\path[draw,thick](1) edge node {} (5);
\path[draw,thick](1) edge node {} (6);
\path[draw,thick](2) edge node {} (5);
\path[draw,thick](2) edge node {} (6);
\path[draw,thick](3) edge node {} (5);
\path[draw,thick](3) edge node {} (6);
\end{tikzpicture} \quad
\begin{tikzpicture}
\node[style={circle,fill=yellow!75!red!40,draw,minimum size=1em,inner sep=1pt]}] (1) at (1*360/6:1.5) {$0.797$};
\node[style={circle,fill=yellow!75!red!40,draw,minimum size=1em,inner sep=1pt]}] (2) at (2*360/6:1.5) {$0.797$};
\node[style={circle,fill=yellow!75!red!40,draw,minimum size=1em,inner sep=1pt]}] (3) at (3*360/6:1.5) {$0.797$};
\node[style={circle,fill=yellow!74!red!40,draw,minimum size=1em,inner sep=1pt]}] (4) at (4*360/6:1.5) {$0.845$};
\node[style={circle,fill=yellow!54!red!56,draw,minimum size=1em,inner sep=1pt]}] (5) at (5*360/6:1.5) {$1.463$};
\node[style={circle,fill=yellow!44!red!64,draw,minimum size=1em,inner sep=1pt]}] (6) at (6*360/6:1.5) {$1.772$};
\path[draw,thick](1) edge node {} (5);
\path[draw,thick](1) edge node {} (6);
\path[draw,thick](2) edge node {} (5);
\path[draw,thick](2) edge node {} (6);
\path[draw,thick](3) edge node {} (5);
\path[draw,thick](3) edge node {} (6);
\path[draw,thick](4) edge node {} (6);
\end{tikzpicture}
\end{center}
 One can see that in the induced graph some vertices have higher energy (those with energy $0.816$).

 \end{exa}

A natural question is whether $d(v_i)<d(v_j)$ implies $\mathcal{E}_G(v_i)<\mathcal{E}_G(v_j)$ where $v_i$ and $v_j$ are arbitrary vertices of a connected graph $G$ . The following example shows that this is not the case.

\begin{exa}
Let $G$ be the graph with vertex set $V=\{1,2,3,4,5,6\}$ which consists of adding a pendant vertex to $K_{2,3}$ .
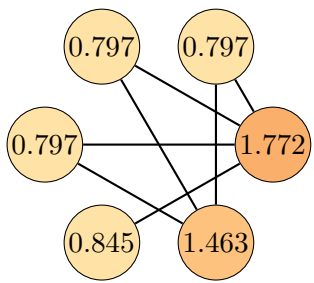
\begin{center}
\begin{tikzpicture}
\node[style={circle,fill=yellow!75!red!40,draw,minimum size=1em,inner sep=1pt]}] (1) at (1*360/6:1.5) {$0.797$};
\node[style={circle,fill=yellow!75!red!40,draw,minimum size=1em,inner sep=1pt]}] (2) at (2*360/6:1.5) {$0.797$};
\node[style={circle,fill=yellow!75!red!40,draw,minimum size=1em,inner sep=1pt]}] (3) at (3*360/6:1.5) {$0.797$};
\node[style={circle,fill=yellow!74!red!40,draw,minimum size=1em,inner sep=1pt]}] (4) at (4*360/6:1.5) {$0.845$};
\node[style={circle,fill=yellow!54!red!56,draw,minimum size=1em,inner sep=1pt]}] (5) at (5*360/6:1.5) {$1.463$};
\node[style={circle,fill=yellow!44!red!64,draw,minimum size=1em,inner sep=1pt]}] (6) at (6*360/6:1.5) {$1.772$};
\path[draw,thick](1) edge node {} (5);
\path[draw,thick](1) edge node {} (6);
\path[draw,thick](2) edge node {} (5);
\path[draw,thick](2) edge node {} (6);
\path[draw,thick](3) edge node {} (5);
\path[draw,thick](3) edge node {} (6);
\path[draw,thick](4) edge node {} (6);
\end{tikzpicture}
\end{center}
In this graph we have a pendant vertex (or leaf) that has larger energy than some vertices with higher degree. Showing the previous conjecture to be false.
\end{exa}

On the other hand there is a graph such that all its vertices have the same energy but not the same degree.

\begin{exa}
The graph with vertex set $V=\{1,2,3,4,5,6\}$ which consists of removing an edge to $K_{3,3}$ an edge is an example of an equienergetic graph that is not regular.

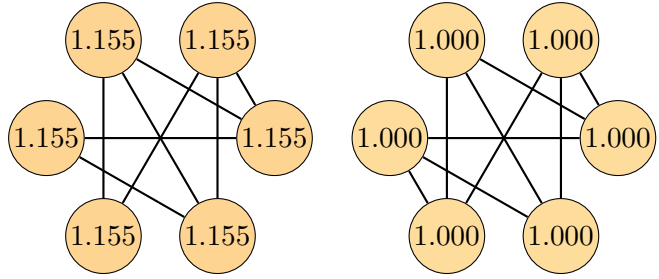
\begin{center}
\begin{tikzpicture}
\node[style={circle,fill=yellow!64!red!48,draw,minimum size=1em,inner sep=1pt]}] (1) at (1*360/6:1.5) {$1.155$};
\node[style={circle,fill=yellow!64!red!48,draw,minimum size=1em,inner sep=1pt]}] (2) at (2*360/6:1.5) {$1.155$};
\node[style={circle,fill=yellow!64!red!48,draw,minimum size=1em,inner sep=1pt]}] (3) at (3*360/6:1.5) {$1.155$};
\node[style={circle,fill=yellow!64!red!48,draw,minimum size=1em,inner sep=1pt]}] (4) at (4*360/6:1.5) {$1.155$};
\node[style={circle,fill=yellow!64!red!48,draw,minimum size=1em,inner sep=1pt]}] (5) at (5*360/6:1.5) {$1.155$};
\node[style={circle,fill=yellow!64!red!48,draw,minimum size=1em,inner sep=1pt]}] (6) at (6*360/6:1.5) {$1.155$};
\path[draw,thick](1) edge node {} (4);
\path[draw,thick](1) edge node {} (5);
\path[draw,thick](1) edge node {} (6);
\path[draw,thick](2) edge node {} (4);
\path[draw,thick](2) edge node {} (5);
\path[draw,thick](2) edge node {} (6);
\path[draw,thick](3) edge node {} (5);
\path[draw,thick](3) edge node {} (6);
\end{tikzpicture} \quad
\begin{tikzpicture}
\node[style={circle,fill=yellow!70!red!44,draw,minimum size=1em,inner sep=1pt]}] (1) at (1*360/6:1.5) {$1.000$};
\node[style={circle,fill=yellow!70!red!44,draw,minimum size=1em,inner sep=1pt]}] (2) at (2*360/6:1.5) {$1.000$};
\node[style={circle,fill=yellow!70!red!44,draw,minimum size=1em,inner sep=1pt]}] (3) at (3*360/6:1.5) {$1.000$};
\node[style={circle,fill=yellow!70!red!44,draw,minimum size=1em,inner sep=1pt]}] (4) at (4*360/6:1.5) {$1.000$};
\node[style={circle,fill=yellow!70!red!44,draw,minimum size=1em,inner sep=1pt]}] (5) at (5*360/6:1.5) {$1.000$};
\node[style={circle,fill=yellow!70!red!44,draw,minimum size=1em,inner sep=1pt]}] (6) at (6*360/6:1.5) {$1.000$};
\path[draw,thick](1) edge node {} (4);
\path[draw,thick](1) edge node {} (5);
\path[draw,thick](1) edge node {} (6);
\path[draw,thick](2) edge node {} (4);
\path[draw,thick](2) edge node {} (5);
\path[draw,thick](2) edge node {} (6);
\path[draw,thick](3) edge node {} (5);
\path[draw,thick](3) edge node {} (6);
\path[draw,thick](3) edge node {} (4);
\end{tikzpicture}
\end{center}

Moreover, since energy of the vertices of $K_{3,3}$ is $1$ this proves that adding an edge can reduce the energy of all vertices in the graph.
\end{exa}
Another natural question is whether the energy sequence of a graph uniquely determines the graph up to isomorphism. This can be seen to be false by the following example.
\begin{exa}
The following three graphs $G_1, G_2$ and $G_3$ (named from left to right)  are $2,3$ and $4$-regular,  respectively, and they have the same vertex energies.  Moreover, notice that $G_1$ is contained in $G_2$ and $G_2$ is contained in $G_3$, so there are subgraphs with the same vertex energies.
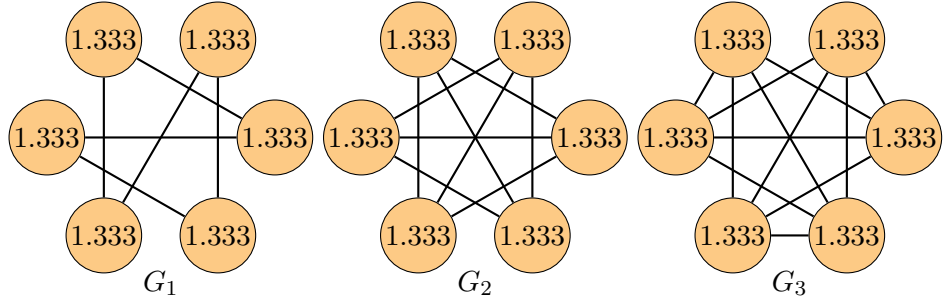
\begin{center}
\begin{tikzpicture}
\node[style={circle,fill=yellow!58!red!53,draw,minimum size=1em,inner sep=1pt]}] (1) at (1*360/6:1.5) {$1.333$};
\node[style={circle,fill=yellow!58!red!53,draw,minimum size=1em,inner sep=1pt]}] (2) at (2*360/6:1.5) {$1.333$};
\node[style={circle,fill=yellow!58!red!53,draw,minimum size=1em,inner sep=1pt]}] (3) at (3*360/6:1.5) {$1.333$};
\node[style={circle,fill=yellow!58!red!53,draw,minimum size=1em,inner sep=1pt]}] (4) at (4*360/6:1.5) {$1.333$};
\node[style={circle,fill=yellow!58!red!53,draw,minimum size=1em,inner sep=1pt]}] (5) at (5*360/6:1.5) {$1.333$};
\node[style={circle,fill=yellow!58!red!53,draw,minimum size=1em,inner sep=1pt]}] (6) at (6*360/6:1.5) {$1.333$};
\path[draw,thick](1) edge node {} (4);
\path[draw,thick](1) edge node {} (5);
\path[draw,thick](2) edge node {} (4);
\path[draw,thick](2) edge node {} (6);
\path[draw,thick](3) edge node {} (5);
\path[draw,thick](3) edge node {} (6);

 \node [label=below:$G_1$] (*) at (0,-1.5  ) {};
\end{tikzpicture}
\begin{tikzpicture}
\node[style={circle,fill=yellow!58!red!53,draw,minimum size=1em,inner sep=1pt]}] (1) at (1*360/6:1.5) {$1.333$};
\node[style={circle,fill=yellow!58!red!53,draw,minimum size=1em,inner sep=1pt]}] (2) at (2*360/6:1.5) {$1.333$};
\node[style={circle,fill=yellow!58!red!53,draw,minimum size=1em,inner sep=1pt]}] (3) at (3*360/6:1.5) {$1.333$};
\node[style={circle,fill=yellow!58!red!53,draw,minimum size=1em,inner sep=1pt]}] (4) at (4*360/6:1.5) {$1.333$};
\node[style={circle,fill=yellow!58!red!53,draw,minimum size=1em,inner sep=1pt]}] (5) at (5*360/6:1.5) {$1.333$};
\node[style={circle,fill=yellow!58!red!53,draw,minimum size=1em,inner sep=1pt]}] (6) at (6*360/6:1.5) {$1.333$};
\path[draw,thick](1) edge node {} (3);
\path[draw,thick](1) edge node {} (4);
\path[draw,thick](1) edge node {} (5);
\path[draw,thick](2) edge node {} (4);
\path[draw,thick](2) edge node {} (5);
\path[draw,thick](2) edge node {} (6);
\path[draw,thick](3) edge node {} (5);
\path[draw,thick](3) edge node {} (6);
\path[draw,thick](4) edge node {} (6);

 \node [label=below:$G_2$] (*) at (0,-1.5  ) {};
\end{tikzpicture}
\begin{tikzpicture}
\node[style={circle,fill=yellow!58!red!53,draw,minimum size=1em,inner sep=1pt]}] (1) at (1*360/6:1.5) {$1.333$};
\node[style={circle,fill=yellow!58!red!53,draw,minimum size=1em,inner sep=1pt]}] (2) at (2*360/6:1.5) {$1.333$};
\node[style={circle,fill=yellow!58!red!53,draw,minimum size=1em,inner sep=1pt]}] (3) at (3*360/6:1.5) {$1.333$};
\node[style={circle,fill=yellow!58!red!53,draw,minimum size=1em,inner sep=1pt]}] (4) at (4*360/6:1.5) {$1.333$};
\node[style={circle,fill=yellow!58!red!53,draw,minimum size=1em,inner sep=1pt]}] (5) at (5*360/6:1.5) {$1.333$};
\node[style={circle,fill=yellow!58!red!53,draw,minimum size=1em,inner sep=1pt]}] (6) at (6*360/6:1.5) {$1.333$};
\path[draw,thick](1) edge node {} (3);
\path[draw,thick](1) edge node {} (4);
\path[draw,thick](1) edge node {} (5);
\path[draw,thick](1) edge node {} (6);
\path[draw,thick](2) edge node {} (3);
\path[draw,thick](2) edge node {} (4);
\path[draw,thick](2) edge node {} (5);

\path[draw,thick](2) edge node {} (6);
\path[draw,thick](3) edge node {} (5);
\path[draw,thick](3) edge node {} (6);
\path[draw,thick](4) edge node {} (5);
\path[draw,thick](4) edge node {} (6);

\node [label=below:$G_3$] (*) at (0,-1.5 ) {};
\end{tikzpicture}

\end{center}
  Finally, the spectra of these graphs are given by  $E_1=\{-2, -1^{(2)}, 1^{(2)}, 2\},$
$E_2=\{-2^{(2)}, 0^{(2)}, 1, 3\} $ and $E_3=\{-2^{(2)}, 0^{(3)}, 4\}.$ So vertex equienergetic does not imply isospectral.
\end{exa}
Finally we provide explicit examples for graphs in which every vertex energy is large. Recall that from Corollary \ref{M_4} we have $\mathcal{E}_G(v_i)\geq \sqrt{\frac{{d_i}^3}{M_4}}$ So if we find a family of regular graphs of arbitrary degree such that $M_4\leq cd^2$ for some $c$ we would have:
\begin{equation}\label{cotacirc}
\mathcal{E}_G(v_i)\geq \sqrt{\dfrac{d_i^3}{cd_i^2}}= \sqrt{\dfrac{d_i}{c}}
\end{equation}

\begin{exa} Consider the $2d$-regular circulant graphs with $4^d+1$ vertices and distances $1,4,\dots,4^{d-1}$, it is not hard to show the only $4$-cycles are those of form $(a,a+e_1 4^i,a+e_1 4^i+ e_2 4^j,a+e_2 4^j,a), (a,a+e_1 4^i,a,a+e_2 4^j,a)$ or $(a,a+e_1 4^i,a+e_1 4^i+ e_2 4^j,a+e_1 4^i,a)$ with $e_1,e_2= \pm 1$, and so $M_4$ is at most $12d^2$. It follows from (\ref{cotacirc}) that the energy of every vertex is at least  $\sqrt\frac{d}{6}$.

\begin{center}
\begin{tikzpicture}[scale=.9]
\node[style={circle,fill=yellow!21!red!83,draw,minimum size=2.5em,inner sep=1pt]}] (1) at (1*360/17:3.2) {$\scriptstyle{1.6}$};
\node[style={circle,fill=yellow!21!red!83,draw,minimum size=2.5em,inner sep=1pt]}] (2) at (2*360/17:3.2) {$\scriptstyle{1.6}$};
\node[style={circle,fill=yellow!21!red!83,draw,minimum size=2.5em,inner sep=1pt]}] (3) at (3*360/17:3.2) {$\scriptstyle{1.6}$};
\node[style={circle,fill=yellow!21!red!83,draw,minimum size=2.5em,inner sep=1pt]}] (4) at (4*360/17:3.2) {$\scriptstyle{1.6}$};
\node[style={circle,fill=yellow!21!red!83,draw,minimum size=2.5em,inner sep=1pt]}] (5) at (5*360/17:3.2) {$\scriptstyle{1.6}$};
\node[style={circle,fill=yellow!21!red!83,draw,minimum size=2.5em,inner sep=1pt]}] (6) at (6*360/17:3.2) {$\scriptstyle{1.6}$};
\node[style={circle,fill=yellow!21!red!83,draw,minimum size=2.5em,inner sep=1pt]}] (7) at (7*360/17:3.2) {$\scriptstyle{1.6}$};
\node[style={circle,fill=yellow!21!red!83,draw,minimum size=2.5em,inner sep=1pt]}] (8) at (8*360/17:3.2) {$\scriptstyle{1.6}$};
\node[style={circle,fill=yellow!21!red!83,draw,minimum size=2.5em,inner sep=1pt]}] (9) at (9*360/17:3.2) {$\scriptstyle{1.6}$};
\node[style={circle,fill=yellow!21!red!83,draw,minimum size=2.5em,inner sep=1pt]}] (10) at (10*360/17:3.2) {$\scriptstyle{1.6}$};
\node[style={circle,fill=yellow!21!red!83,draw,minimum size=2.5em,inner sep=1pt]}] (11) at (11*360/17:3.2) {$\scriptstyle{1.6}$};
\node[style={circle,fill=yellow!21!red!83,draw,minimum size=2.5em,inner sep=1pt]}] (12) at (12*360/17:3.2) {$\scriptstyle{1.6}$};
\node[style={circle,fill=yellow!21!red!83,draw,minimum size=2.5em,inner sep=1pt]}] (13) at (13*360/17:3.2) {$\scriptstyle{1.6}$};
\node[style={circle,fill=yellow!21!red!83,draw,minimum size=2.5em,inner sep=1pt]}] (14) at (14*360/17:3.2) {$\scriptstyle{1.6}$};
\node[style={circle,fill=yellow!21!red!83,draw,minimum size=2.5em,inner sep=1pt]}] (15) at (15*360/17:3.2) {$\scriptstyle{1.6}$};
\node[style={circle,fill=yellow!21!red!83,draw,minimum size=2.5em,inner sep=1pt]}] (16) at (16*360/17:3.2) {$\scriptstyle{1.6}$};
\node[style={circle,fill=yellow!21!red!83,draw,minimum size=2.5em,inner sep=1pt]}] (17) at (17*360/17:3.2) {$\scriptstyle{1.6}$};
\path[draw,thick](1) edge node {} (2);
\path[draw,thick](1) edge node {} (5);
\path[draw,thick](1) edge node {} (14);
\path[draw,thick](1) edge node {} (17);
\path[draw,thick](2) edge node {} (3);
\path[draw,thick](2) edge node {} (6);
\path[draw,thick](2) edge node {} (15);
\path[draw,thick](3) edge node {} (4);
\path[draw,thick](3) edge node {} (7);
\path[draw,thick](3) edge node {} (16);
\path[draw,thick](4) edge node {} (5);
\path[draw,thick](4) edge node {} (8);
\path[draw,thick](4) edge node {} (17);
\path[draw,thick](5) edge node {} (6);
\path[draw,thick](5) edge node {} (9);
\path[draw,thick](6) edge node {} (7);
\path[draw,thick](6) edge node {} (10);
\path[draw,thick](7) edge node {} (8);
\path[draw,thick](7) edge node {} (11);
\path[draw,thick](8) edge node {} (9);
\path[draw,thick](8) edge node {} (12);
\path[draw,thick](9) edge node {} (10);
\path[draw,thick](9) edge node {} (13);
\path[draw,thick](10) edge node {} (11);
\path[draw,thick](10) edge node {} (14);
\path[draw,thick](11) edge node {} (12);
\path[draw,thick](11) edge node {} (15);
\path[draw,thick](12) edge node {} (13);
\path[draw,thick](12) edge node {} (16);
\path[draw,thick](13) edge node {} (14);
\path[draw,thick](13) edge node {} (17);
\path[draw,thick](14) edge node {} (15);
\path[draw,thick](15) edge node {} (16);
\path[draw,thick](16) edge node {} (17);
\end{tikzpicture}

The proposed graph for $d=2$.
\end{center}
\end{exa}

\section{Koolen-Moulton Type inequalities}\label{Koolen-MoultonIN}

After McClelland's inequality the next step in giving bounds for the energy in terms of the number of vertices and the number edges was given by Koolen and Moulton. In their paper \cite{KM}, they improved the upper bound of McClelland (\ref{Equa1}), by showing that for any graph $G$ with $n$ vertices and $m$ edges such that $2m \geq n$, the following inequality holds,
\begin{equation}
   \mathcal{E}(G) \leq \frac{2m}{n}+\sqrt{(n-1)\left[2m-\left(\frac{2m}{n}\right)^{2}\right]}.
\end{equation}
In this section we present a new inequality, which is an analog of the Koolen-Moulton Inequality, but in this case for the energy of vertex.  The first step is to give a bound in terms of the maximal eigenvalue analog to  \cite[eq. 3]{KM}, but including the weights $p_{ij}$ from (\ref{W1}).

\begin{prop} Let $G$ be a connected graph on $n\geq 2$ vertices, $d_{i}$ the degree of the vertex $v_{i}$ and $\Delta = \max(d_{j}), j=1,\dots,n$, then the following inequality holds
\begin{equation}\label{KM0}
 \mathcal{E}_{G}(v_i)\leq p_{i1}|\lambda_{1}|+ \sqrt{\left(d_{i} - p_{i1}|\lambda_{1}|^{2}\right)(1-p_{i1})}.
\end{equation}
\end{prop}

\begin{proof}

To obtain this new inequality we will use the Cauchy-Schwarz Inequality
\begin{equation}
  \left(\displaystyle\sum_{i= 1} ^ {N} a_{i}b_{i}\right)^{2}  \leq \left(\displaystyle\sum_{i= 1} ^ {N} a_{i}^{2}\right) \left(\displaystyle\sum_{i= 1} ^ {N} b_{i}^{2}\right),
\end{equation}
which holds for arbitrary real-valued numbers $ a_{i},b_{i}$ and $i=1,2,\ldots,N$.\\

Choosing  $N=n-1$, $ a_{j} =\sqrt{p_{ij+1}}|\lambda_{j +1}| $, $ b_{i} =\sqrt{p_{ij+1}}$ for $j=1,\dots,{n-1}$  we have that
\begin{equation}\label{Equa4}
	\left(\mathcal{E}_{G}(v_i) - p_{i1}|\lambda_{1}| \right) ^{2} \leq \left(d_{i} - p_{i1}|\lambda_{1}|^{2}\right)(1-p_{i1}),
\end{equation}
which, taking square roots and  rearranging, yields
\begin{equation}\label{Equa5}
	\mathcal{E}_{G}(v_i)\leq p_{i1}|\lambda_{1}|+ \sqrt{\left(d_{i} - p_{i1}|\lambda_{1}|^{2}\right)(1-p_{i1})}.
\end{equation}

\end{proof}

In order to get an inequality of Koolen-Moulton type, we need to bound the weights that are given in (\ref{W1}) in terms of combinatorial quantities, namely the maximum degree and the eccentricity of the vertex.

\begin{prop}
	Let $G$ be a connected graph on $n\geq 2$ vertices and $v_{i}$ a vertex. Suppose $G$ has  eigenvalues $\lambda_1,\lambda_2,\dots, \lambda_n$ and associated weights $p_{i1},p_{i2},\dots,p_{in}$ such that $|\lambda_j|\geq |\lambda_{j+1}|$, and $p_{ij}\geq p_{ij+1}$ if $|\lambda_j| = |\lambda_{j+1}|$, then $p^i_1\geq \frac{1}{\lambda_1^{2r}+\lambda_1^{2r-2}}$ where $r$ is the eccentricity of $G$ at $v_i$.
\end{prop}

\begin{proof}

Let $C_k$ be the set of all closed walks of length $k$ in $G$ and let $C_{k}(v_i)$ be the set of all closed $v_i-v_i$ walks in $G$ of length $k+2r-2$ or $k+2r$, where $r$ is the eccentricity of $v_i$. We give an explicit injection from $C_{k}(v_i)$ to $C_k$. Thus proving $|C_k|\leq |C_{k}(v_i)|$.

	The injection $f:C_k\to C_{k(v_i)}$ is as follows: Let $W$ be a closed walk in $G$ starting at vertex $w$. Take a path $P$ of length $r$ or $r-1$ from $v_i$ to $w$, such that the length of $P$ has the same parity as  the distance between $v_i$ and $w$. We define $f(W)$ as the concatenation of $P$, $W$ and $\overline{P}$. It is clear that $f(W)$ is a $v-v$ path of length $k+2r$ or $k+2r-2$.

The injectivity follows since we can recover $W$ from $f(W)$ by first recovering $w$:  $w$ is the $r+1$'th vertex in the walk if $f(W)$ has length $k+2r$ and it is the $r$'th vertex in the walk if $f(W)$ has length $k+2r-2$.  $W$ is given by $w$ together with the next $k-1$ steps of $f(W)$.

	Recall that the number of closed walks in $G$ of length $k$ is given by $M_k=\sum_{j=1}^n \lambda_j^k$ and the number of closed $v_i-v_i$ walks of length $k$ is given by $\sum_{j=1}^n p_{ij}\lambda_j^k.$

	Thus the inequality $|C_k|\leq |C_{k}(v_i)|$ may be written as $$ \sum\limits_{j=1}^n (\lambda_j^k)\leq\sum\limits_{j=1}^n p_{ij}\lambda_j^{k+2r}+\sum\limits_{j=1}^n p_{ij}\lambda_j^{k+2r-2}$$ which implies \begin{equation} \label{ineqw} \frac{\sum\limits_{j=1}^n p^i_j(\lambda_j^{k+2r}+\lambda_j^{k+2r-2})}{\sum\limits_{j=1}^n \lambda_j^k}\geq 1.\end{equation}

 Finally, suppose that $|\lambda_1|$ has multiplicity $l$ for $|A(G)|$, the absolute value of the adjacency matrix. That is,  $|\lambda_1|=|\lambda_2|=\dots=|\lambda_l|$. Then, since $|\lambda_i|<|\lambda_1|$ for $i>l$,  taking the limit in (\ref{ineqw}), as $k$ approaches infinity on the even numbers yields
	\begin{eqnarray*} 1&\leq &\lim_{k\to \infty}\frac{\sum\limits_{j=1}^n p_{ij}(\lambda_j^{k+2r}+\lambda_j^{k+2r-2})}{\sum\limits_{j=1}^n \lambda_j^k} =\lim_{k\to \infty}\frac{\sum\limits_{j=1}^l p_{ij}(\lambda_j^{k+2r}+\lambda_j^{k+2r-2})}{\sum\limits_{j=1}^l \lambda_j^k}\\
	&=& \frac{(p_{i1}+p_{i2}+\dots+p_{il})(\lambda_1^{2r}+\lambda_1^{2r-2})}{l},\end{eqnarray*}
	which implies $p_{i1}\geq \frac{1}{\lambda_{1}^{2r}+\lambda_{1}^{2r-2}}$.

\end{proof}

Now we are ready to prove the desired inequality. For this, let $\alpha=\max\left(\sqrt {\sum\limits_{j=1}^n \frac{d_j^2 }{n} },\sqrt{\Delta}\right)$. We will use the bound $\lambda_1\geq \alpha$  from \cite{Zhou2}\footnote{Most of the times we have $\sqrt {\sum\limits_{i=1}^n \frac{d_i^2 }{n} } \geq \sqrt{\Delta}$, but not always, for a counterexample consider a large tree consisting of one vertex of degree $d>2$, $d$ vertices of degree $1$, and the remaining vertices of degree $2$, in such a tree we have $\sqrt {\sum\limits_{i=1}^n \frac{d_i^2 }{n} }\approx 2$}. On the other hand we have $\lambda_1\leq \Delta$ , which gives the bound $$p_{i1}\geq \frac{1}{\Delta^{2r}+\Delta^{2r-2}}\geq\frac{1}{2\Delta^{2r}}.$$

\begin{thm} Let $G$ be a connected graph on $n\geq 2$ vertices, $d_{i}$ the degree of the vertex $v_{i}$, $\Delta = \max(d_{j}), j=1,\dots,n$ and  $r$ the radius of $G$ at $v_j$. Then the inequality
\begin{equation}\label{KM1}
	\mathcal{E}_G(v_{j}) \leq\frac{1}{2\Delta^{2r}}|\alpha|+\sqrt{\left(d_{i}-\frac{1}{2\Delta^{2r}}\alpha^2\right)\left(1-\frac{1}{2\Delta^{2r}}\right)}.
\end{equation}
holds. Moreover, equality holds in (\ref{KM1}) if and only if $G$ is $K_2$.
\end{thm}
\begin{proof}
From the previous proposition we have
\begin{equation}
   \mathcal{E}_{G}(v_j)\leq p_{i1}|\lambda_{1}|+ \sqrt{\left(d_{i} - p_{i1}|\lambda_{1}|^{2}\right)(1-p_{i1})}.
\end{equation}

Now, since the function $f(x,p)=px+\sqrt{(d_i-px^2)(1-p)}$ decreases in $x$ in the interval $\sqrt{d_{i}}\leq x \leq \sqrt{\frac{d_{i}}{p_{i1}}}$, in view of $p_{i1} \leq 1$, and increases in $p$.

From here, we have
\begin{eqnarray*}
  \mathcal{E}_G(v_{i}) &\leq & p_{i1}|\alpha|+\sqrt{(d_{i}-p_{i1}\alpha^2)(1-p_{i1})}, \\
   &\leq &\frac{1}{2\Delta^{2r}}|\alpha|+\sqrt{\left(d_{i}-\frac{1}{2\Delta^{2r}}\alpha^2\right)\left(1-\frac{1}{2\Delta^{2r}}\right)}.
\end{eqnarray*}
Equality holds only if $\Delta^{2r}=\Delta^{2r-2}$, which implies that $\Delta=1.$\end{proof}

\begin{rem}

	\begin{enumerate} \item The bound given by the Koolen-Moulton type inequality is $$\mathcal{E}_G(v_i)\leq\frac{1}{2\Delta^{2r}}|\alpha|+\sqrt{\left(d_{i}-\frac{1}{\Delta^{2r}+\Delta^{2r-2}}\alpha^2\right)\left(1-\frac{1}{\Delta^{2r}+\Delta^{2r-2}}\right)}.$$ This expression is always strictly less than $\sqrt{d}$ unless $\lambda_1=\sqrt{d_i}$, so this bound is indeed an improvement of McCleland's inequality for vertex energy.
\item
For $d$-regular graphs we have $\lambda_1=d$ and we have $p_1^i=\frac{1}{n}$ for every vertex $v_i$. Thus we obtain the bound $$\mathcal{E}_G(v_i)\leq \frac{1}{n}d+\sqrt{(d-\frac{d^2}{n})(1-\frac{1}{n})}= \sqrt{d}\frac{\sqrt{d}+ \sqrt{(n-d)(n-1)} }{n}.$$ If we sum over all vertices we get the exact same bound that is obtained by using the regular Koolen-Moulton inequality for graph energy: $\mathcal{E}(G)\leq \frac{dn}{n}+\sqrt{(n-1)(dn-(\frac{dn}{n})^2)}$
\item Note that for the complete graph $K_n$ we obtain $$\mathcal{E}_G(v_i)\leq \sqrt{n-1}\frac{\sqrt{n-1} + \sqrt{n-1}}{n}=\frac{2n-2}{n}$$ which is sharp.
\end{enumerate}
\end{rem}

\section{Locally Finite graphs}\label{LFG}

The energy of a graph can only be defined for finite graphs. However,  we introduce the definition of the energy of a vertex for uniformly locally finite graphs.
In order to define the energy of a vertex for this graphs we borrow the framework and tools of non commutative probability and extend some of the above results to this setting.  While a priori the energy of a vertex for this type of graphs may be difficult to interpret for the perspective of mathematical chemistry, extending the framework give a broader picture of this quantity, since as we will see, it has some consistent continuity properties.

\subsection{Non Commutative Probability}
In this section we present an introduction of the usual concepts in Non Commutative Probability. For more details on this topic we recommend to read \cite{NiSp}, or \cite{HoOb} for a setting closer to the one used in this section.

A $C^*$\textit{-probability space} is a pair $(\mathcal{A},\varphi)$, where $\mathcal{A}$ is a unital $C^*$-algebra and $\varphi:\mathcal{A}\to\mathbb{C}$ is a positive unital linear functional such that $\varphi(1)=1$. The elements of $\mathcal{A}$ are called (non-commutative) random variables. An element $a\in\mathcal{A}$ such that $a=a^*$ is called self-adjoint. The functional $\varphi$ should be understood as the expectation in classical probability.

For any self-adjoint element $a\in\mathcal{A}$ there exists a unique probability measure $\mu_a$ (its distribution) with the same moments as $a$, that is, $$\int_{\mathbb{R}}x^{k}\mu_a (dx)=\varphi (a^{k}), \quad \forall k\in \mathbb{N}.$$

We say that a sequence $a_n\in\mathcal{A}_n$ \emph{converges in distribution} to $a\in\mathcal{A}$ if $\mu_{a_n}$ converges in distribution to $\mu_a$.
In this setting convergence in moments is more often used than convergence in distribution.

Let $(\varphi_n,\mathcal{A}_n)$ be a sequence of $C^*$-probability spaces and let $a\in(\mathcal{A},\varphi)$ be a self-adjoint random variable. We say that the sequence $a_n\in(\varphi_n,\mathcal{A}_n)$ of self-adjoint random variables converges to $a$ \textit{in moments} if
$$\lim_{n\to\infty}\varphi_n(a_n^k)=\varphi(a^k)\ \text{for all}\ k\in\mathbb{N}.$$

If $a$ is bounded then convergence in moments implies convergence in distribution.

\subsection{Adjacency Algebra}

For our purposes  we will consider the \textit{adjacency algebra} of a uniformly locally finite graph $G$ and a vacuum state $\varphi_i$.

Specifically, for $x\in V$, let $\delta(x)$ be the indicator function of the one-element set $\{x\}$. Then $\{\delta(x),\,~x\in V\}$ is an orthonormal basis of the Hilbert space $l^{2}(V)$ of square integrable functions on the set $V$, with the usual inner product.

We identify $A$ with the densely defined symmetric operator on $l^{2}(V)$ defined by
\begin{equation}
A\delta(x)=\sum_{x\sim x'}\delta(x')
\end{equation}
for $x\in V$. Notice that the sum on the right-hand-side is finite since our graph is assumed to be locally finite. It is known that $A(\mathcal{G})$ is bounded iff ${\mathcal G}$ is uniformly
locally finite and the spectral radius is bounded by $\Delta$. If $A(\mathcal{G})$ is essentially self-adjoint, its closure is called the
\textit{adjacency operator} of ${\mathcal G}$ and its spectrum is called the spectrum of ${\mathcal G}$.

The unital algebra generated by $A$, i.e. the algebra of polynomials in $A$, is called the \textbf{adjacency algebra} of ${\mathcal G}$ and is denoted by ${\mathcal A}({\mathcal G})$ or simply by ${\mathcal A}$.

For an element $T\in{\mathcal A}({\mathcal G})$ the vacuum state $\varphi_i$ evaluated in $T$ is given by $$\varphi_i(T)=\left<e_1T, e_i\right>.$$

As in the case where $G$ is finite the $k$-th moment of $A$ with respect to the $\varphi_{i}$  is given the the number of walks in $G$ of size $k$ starting and ending at the vertex $i$.
That is, $$\varphi_{i}(A^k)=|\{(v_1,...,v_k): v_1=v_k=1 ~ and ~ (v_i,v_{i+1})\in E\}|.$$

\subsection{Energy of a vertex for locally finite graph}

For any bounded operator $A$ in some $C^{*}$-algebra we can define the absolute value of $A$ by the formula $|A|:=\sqrt{AA^{*}}.$
Since $\phi_i$ is positive we may use the inequalities (\ref{Equa197}) and (\ref{Equa198}) from Section \ref{EV}  which read as follows: For all $i\in\{1,\dots,n\}$
\begin{equation}
\phi_i(|A|)^2\leq \phi_i(|A|^2).
\end{equation}
and
\begin{equation}
\phi_i(|A|)\geq\phi_i(A^2)^{3/2} /\phi_i(A^4)^{1/2}.
\end{equation}
Thus we arrive to generalizations of Proposition \ref{McClelland vertex} and Corollary \ref{M_44}

\begin{thm}
For any locally finite graph $G$ and any vertex $v_i$, the following inequalities hold
\begin{equation}
\cfrac{(d_i)^{3/2}} {M_4(G,i)^{1/2}}\leq\mathcal{E}_G(v_i)\leq \sqrt{d_i}.
\end{equation}
\end{thm}
Similarly, for uniformly locally finite graphs we obtain the generalization of \eqref{lowerbound2}.
\begin{thm}Let $G$ be a  graph with at least one edge. Then
\begin{equation}
\mathcal{E}_G(v_i)\geq \sqrt{ \cfrac{d_i}{\Delta}} , \quad\quad~~~\text{for all } v_i\in V.
\end{equation}
\end{thm}

We now state a continuity theorem for uniformly locally finite graphs.

\begin{thm} \label{continuity} Let $G_n$ be a sequence of locally finite graphs and for each $n$, let $v_i^{(n)}$  be a given vertex in $G_n$. Moreover let $G$ be a uniformly locally finite graph and $v_i\in G$. If for every $d$ there exists a $N(d)$ such that the neighborhood of radius $d$ around $v^{(n)}_i$ is isomorphic (as graphs) to the neighborhood of radius $d$ around $v_i$ then
\begin{equation}
\mathcal{E}_{_n}(v^{(n)}_i)\to \mathcal{E}_G(v_i)
\end{equation}
as $n\to\infty$.
\end{thm}
\begin{proof} By an application of Stone-Weierstrass, convergence of all moments is enough to ensure the convergence of all continuous functions, in particular, the absolute value. The convergence of moments is readily implied by the condition that the neighborhood of radius $d$ around $v^{(n)}_i$ is isomorphic (as graphs) to the neighborhood of radius $d$ around $v_i$, since the $d$-th moments coincide for $n>N(d)$.
\end{proof}

We finish with some examples of uniformly locally finite graphs and their approximation.

\subsection*{The line}

Consider the graph $G=(\mathbb{Z},E)$ where $E=\{i\sim (i+1)|i\in\mathbb{Z}\}$.

The $k$-th moment with respect to any vertex is $0$ for odd $k$ and $\binom{k}{k/2}$ when $k$ is even. From this we recover the spectral distribution w. r. t. any vertex; see for example \cite[Example 1.14] {NiSp}. This is given by the arcsine law of variance $2$, with density
\begin{equation}\label{arcsine} \mathbf{d}\mu_\mathbf{a}=\frac{1}{\pi \sqrt{4-x^2}}\mathbf{d}x.\end{equation}
The energy of any vertex is then calulated by the first absolute moment of this density,
\begin{equation}\int^{\sqrt{2}}_{-\sqrt{2}} |x|\frac{1}{\pi \sqrt{4-x^2}}\mathbf{d}x=\frac{4}{\pi}.\end{equation}

Given the continuity theorem above, Theorem $\ref{continuity}$, it is now not a surprise that this coincides with the limits given by \eqref{limit1} and \eqref{limit3}.

\subsection*{The semiline}

Consider now the graph $G=(\mathbb{N},E)$ where $E=\{i\sim (i+1)|i\in\mathbb{N}\}$. In this case the vertex energy actually depends on the choice of the vertex. First, the $k$-th moment, with respect to the vertex 1, is $0$ for odd $k$ and the Catalan number $\frac{1}{k+1}\binom{k}{k/2}$ when $k$ is even. Thus,  spectral distribution  r. t. of the vertex $1$ is given by the semicircle law of variance $1$, with density
\begin{equation}\mathbf{d}\mu_\mathbf{s}=\frac{1}{2\pi}\sqrt{4-x^2}\mathbf{d}x.\end{equation}
Again, the energy of any vertex is then calculated by the first absolute moment,
\begin{equation}\int_{2}^{2} |x|\frac{1}{2\pi}\sqrt{4-x^2}\mathbf{d}x=\frac{8}{3\pi}.\end{equation}
Theorem $\ref{continuity}$ again ensures the accordance with the limit \eqref{limit2}, for $i=1$.

The spectral distribution of w. r. t. of the vertex $i$ could be calculated by means of non-commutative probability \cite{HoOb}, and, in principle, from this information one can obtain the energy. Instead, we use Theorem $\ref{continuity}$ to calculate the energy from \eqref{limit3}, obtaining $$\frac{4}{\pi}+\frac{4(-1)^i}{\pi(4i^2-1)}.$$

This quantity is maximized for $i=2$, which coincides with the same empirical observation for the path $P_n$.

\subsection*{The $d$-regular tree.}
Let $d\geq2$, the $d$-regular tree is the unique tree with all vertices of degree $d$. The case $d=2$ is the line, treated above, so we assume $d>2$. The spectral distribution of the $d$-regular tree w. r. t. any vertex is given by the Kesten-Mckay distribution.
For $d\geq 2$, an integer, the Kesten-McKay distribution, $\mu_d$, is the measure with density
\begin{equation}\label{kestendistribution} \mathbf{d}\mu_d=\frac{d\sqrt{4(d-1)-x^2}}{2\pi(d^2-x^2)}\mathbf{d}x.\end{equation}

By a similar calculation as above, the energy of any vertex is given by \begin{equation}\label{ab} \frac{1}{\pi} \Big(2d
\sqrt{d-1}-d(d-2)\arctan \frac{2\sqrt{d-1}}{d-2}\Big)\end{equation}
which for $d$ large approximates $\frac{8}{3\pi}\sqrt{d}$.\footnote{$\frac{8}{3\pi}\sqrt{d+0.2}$ actually approximation with an error less than $0.01$ for $d\geq5$.}
Now, McKay's theorem states that $d$-regular graphs with few cycles converge in distribution to the $d$-regular tree.

\begin{thm}\cite{McK} \label{d-regular graphs}
 For $d\geq 2$ fixed, let $X_1,\ X_2,\ \dots$ be a sequence of $d$-regular graphs such that $n(X_i)\to\infty$ and
 $c_j(X_i)/n(X_i)\to 0$ as $i\to\infty$ for each $j\geq 3$. Then the distribution with respect to the normalized trace
 of $(X_i)$ converges in moments, as $i\to \infty$, to the Kesten-McKay distribution.
\end{thm}

The above theorem implies that under the above assumptions a typical vertex has energy close to \eqref{ab}.

\section{Conclusions}

In this work we provided many properties for the energy of a vertex. The theory is of course not complete and we hope this work opens some new research lines in the study of the graph energy.  Similar quantities may be defined for the Laplacian energy \cite{GZ}, or ABC energy \cite{ER3}. Interestingly, for the Estrada index \cite{ER2} the analog of the energy of a graph has been considered in \cite{ER} as a centrality measure for the vertex $i$.  It would be interesting to investigate the vertex energy from the viewpoint of centrality.

In other direction, as pointed out by Nikiforov \cite{Nik,Nik2}, the energy of random graphs either for an Erdos-Renyi graph or for uniform $d$-regular graphs may be deduced easily from well-known results from Random Matrix Theory. This results may be translated directly to a \textbf{typical} vertex in a random graph. However, more interesting questions appear when considering properties for all vertices. In this direction we have the following conjectures.

\begin{conj}
(1) A uniform random graph on $n$ points is asymptotically almost surely not completely hyperenergetic.\\
(2) A $d$-regular graph is almost surely completely hyperenergetic for $d\geq7.$
\end{conj}

 Finally,  a natural question following the above program is to give a Coulson Integral Formula \cite{Coul}. This is done in an ongoing project by the first author in joint work with B. Luna and M. Ram\'irez.

\end{document}